\documentclass[10pt]{article}

\usepackage{amssymb,latexsym,amsmath,enumerate,verbatim,amsfonts,amsthm,mathabx}
\usepackage{footmisc}

\usepackage{algorithm,algorithmic,cite,multicol,microtype,bm}
\usepackage{multirow}
\usepackage[figuresright]{rotating} 
\usepackage{color} 
\usepackage{graphicx}
\usepackage{subcaption}
\usepackage{mathtools}
\usepackage{blkarray, bigstrut} %

\usepackage[active]{srcltx}
\usepackage[colorlinks,
            linkcolor=blue,
            anchorcolor=blue,
            citecolor=blue]{hyperref}

\allowdisplaybreaks[4]
\numberwithin{equation}{section}

\usepackage[framemethod=tikz]{mdframed}
\mdfsetup{%
  skipbelow=4pt,
  skipabove=8pt,
  linewidth=1.25pt,
  backgroundcolor=gray!10,
  userdefinedwidth=\textwidth,
  roundcorner=10pt,
}

\usepackage[colorinlistoftodos]{todonotes}
\definecolor{blue}{rgb}{0.0, 0.0, 0.0} 

\newtheorem{lemma}{Lemma}[section]
\newtheorem{definition}{Definition}[section]

\newtheorem{theorem}{Theorem}[section]

\newtheorem{proposition}{Proposition}[section]
\newtheorem{remark}{Remark}[section]

\newtheorem{assumption}{Assumption}[section]

\def\R{{\rm I\!R}}

\def\Argmin{\mathop{\rm Arg\,min}}

\def\sgn{{\rm sgn}}

\title{\sf Doubly majorized algorithm for sparsity-inducing optimization problems with regularizer-compatible constraints}

\author{
Tianxiang Liu \thanks{Department of Mathematical and Computing Science, School of Computing, Tokyo Institute of Technology, Tokyo, Japan (\texttt{liu@c.titech.ac.jp}). This author's research is supported in part by ACT-X, Japan Science and Technology Agency (Grant No. JPMJAX210Q). }
\and
Ting Kei Pong \thanks{Department of Applied Mathematics, Hong Kong Polytechnic University, Hong Kong, People's Republic of China (\texttt{tk.pong@polyu.edu.hk}). This author's research is supported in part by Hong Kong Research Grants Council PolyU153003/19p.}
\and
Akiko Takeda \thanks{Department of Creative Informatics, Graduate School of Information Science and Technology, University of Tokyo, Tokyo, Japan,
  and RIKEN, Center for Advanced Intelligence Project, 
  Tokyo, Japan (\texttt{takeda@mist.i.u-tokyo.ac.jp}, \texttt{akiko.takeda@riken.jp}). This author's research is supported in part by
JSPS KAKENHI Grant Number 19H04069.}
}

\begin{document}

\maketitle

\begin{abstract}
  We consider a class of sparsity-inducing optimization problems whose constraint set is regularizer-compatible, in the sense that, the constraint set becomes easy-to-project-onto after a coordinate transformation induced by the sparsity-inducing regularizer. Our model is general enough to cover, as special cases, the ordered LASSO model in \cite{TibSuo16} and its variants with some commonly used nonconvex sparsity-inducing regularizers. The presence of both the sparsity-inducing regularizer and the constraint set poses challenges on the design of efficient algorithms. In this paper, by exploiting absolute-value symmetry and other properties in the sparsity-inducing regularizer, we propose a new algorithm, called the Doubly Majorized Algorithm (DMA), for this class of problems. The DMA makes use of projections onto the constraint set after the coordinate transformation in each iteration, and hence can be performed efficiently. \emph{Without} invoking any commonly used constraint qualification conditions such as those based on horizon subdifferentials, we show that any accumulation point of the sequence generated by DMA is a so-called $\psi_{\rm opt}$-stationary point, a new notion of stationarity we define as inspired by the notion of $L$-stationarity in \cite{BeckEldar13,BeckHallak16}. We also show that any global minimizer of our model has to be a $\psi_{\rm opt}$-stationary point, again without imposing any constraint qualification conditions. Finally, we illustrate numerically the performance of DMA on solving variants of ordered LASSO with nonconvex regularizers.
\end{abstract}

\section{Introduction}

Sparsity structures arise frequently in contemporary applications such as compressed sensing \cite{CanRomTao06,CanTao05,Chartrand07,ChenDonohoSaunders01} and variable selections \cite{HasTibFri08,Tib96}. In these scenarios, typically, one attempts to find a {\em sparse} vector $\hat x$ such that $A\hat x \approx b$, where $A\in \R^{m\times n}$ and $b\in \R^m$ are given. The corresponding optimization problem can be formulated as
\begin{equation}\label{Prob0}
    \min\limits_{x\in \R^n} \ \ \frac12 \|Ax - b\|^2 + \lambda \sum_{i=1}^n\theta(|{\color{blue}x_i}|),
\end{equation}
where $\theta:\R_+\to \R_+$ is a {\em sparsity-inducing} function, and $\lambda > 0$ is a parameter trading off data fidelity and sparsity in $x$. Popular choices of $\theta$ include:
\begin{enumerate}[(i)]
  \item $\theta(t) = t$ --- The corresponding \eqref{Prob0} then becomes the LASSO model considered in \cite{Tib96};
  \item $\theta(t) = t^p$ for some $p \in (0,1)$ --- The corresponding \eqref{Prob0} belongs to the class of bridge regression models considered in \cite{KnightFu00};
  \item $\theta(t) = \log(1+t/\epsilon)$ for some $\epsilon > 0$ --- This choice of $\theta$ was used in \cite{CanWakBoyd08} for enhancing the sparsity-inducing property of the LASSO model.
\end{enumerate}
Notice that the optimization problem \eqref{Prob0} corresponding to $\theta(t) = t$ is convex, while the problems associated with the other two choices of $\theta$ are nonconvex in general.
Efficient algorithms for solving \eqref{Prob0} with the above choices of $\theta$ abound in the literature. Many of them leverage the computation of the so-called proximal mapping of $\gamma \theta(|\cdot|)$, $\gamma > 0$, i.e., for each $s$, compute a minimizer of the function $t\mapsto \frac1{2\gamma}(s - t)^2 + \theta(|t|)$. We refer the readers to \cite{BeckTeboulle09,gist13,ZengLinWangXu14} and references therein for efficient algorithms for \eqref{Prob0} with the above choices of $\theta$.

While model \eqref{Prob0} makes use of the function $\theta$ to induce sparsity in its solution, it does not explicitly take into account of other structures that may be present in the desired solution. Prior information on these other desirable structures can be incorporated by additionally requiring $x$ to lie in a certain closed set modeling the structures. One recent example is the ordered LASSO model in \cite{TibSuo16} that arises when considering regression problems with time lag, where there is a natural ordering in the magnitude of $x_i$. The basic optimization model takes the following form:
\begin{equation}\label{Prob1}
\begin{array}{rl}
  \min\limits_{x\in \R^n}  & \displaystyle\frac12 \|Ax - b\|^2 + \lambda \sum_{i=1}^n|x_i|\\
  {\rm s.t.} & |x_1|\ge |x_2| \ge \cdots \ge |x_n|,
\end{array}
\end{equation}
where $\lambda > 0$, $A\in \R^{m\times n}$ and $b\in \R^m$; this problem is a variant of \eqref{Prob0} with $\theta(t) = t$ and an {\em additional} constraint
\begin{equation}\label{constraint}
|x_1|\ge |x_2| \ge \cdots \ge |x_n|.
\end{equation}
Notice that one may also replace the sparsity-inducing function in \eqref{Prob1} (i.e., $\theta(t) = t$) by other {\em nonconvex} sparsity-inducing functions such as $\theta(t) = t^p$ ($0<p<1$) and $\theta(t) = \log(1+t/\epsilon)$ ($\epsilon > 0$) as described before, which typically have better empirical sparsity-inducing performances; see, for example, \cite{CanWakBoyd08,Chartrand07}.

Unlike \eqref{Prob0} which can be solved efficiently via the proximal gradient algorithm and its variants for many commonly used $\theta$, with the additional constraint \eqref{constraint}, it is not immediately clear how \eqref{Prob1} (and its variants with different $\theta$) can be solved efficiently; this is especially true when some nonconvex sparsity-inducing regularizers such as $\theta(t) = t^p$ ($0<p<1$) are adopted. One way to get around is to {\em approximate} \eqref{Prob1} by the following {\em convex} optimization problem, as suggested in \cite{TibSuo16}:
\begin{equation}\label{Prob11}
\begin{array}{rl}
  \min\limits_{y,z\in \R^n}  & \displaystyle\frac12 \|A(y - z) - b\|^2 + \lambda \sum_{i=1}^n(y_i + z_i)\\
  {\rm s.t.} & y_1\ge y_2 \ge \cdots \ge y_n \ge 0,\\
  & z_1\ge z_2 \ge \cdots \ge z_n \ge 0.
\end{array}
\end{equation}
Note that $y$ and $z$ are confined to be in the so-called isotone cone, whose projections can be computed efficiently via the classical pool-adjacent-violators algorithm (PAVA). Thus, model \eqref{Prob11} can be solved efficiently via the gradient projection algorithm and its variants. However, this approach may be compromised in terms of interpretability of the solution obtained. Furthermore, in the case when a nonconvex sparsity-inducing regularizer is adopted in place of the $\ell_1$ norm in \eqref{Prob1} for inducing sparser solutions, this approximation technique no longer leads to convex models that admit efficient algorithms.

In this paper, we consider a general optimization model that covers \eqref{Prob1} and some of its variants based on nonconvex sparsity-inducing regularizers as special cases, and develop an algorithm for solving this class of problems. Specifically, we consider the following optimization problem:
\begin{equation}\label{Prob}
  \min\limits_{x\in \R^n} \ \ F(x) := f(x) + \lambda\Psi(|x|) + \delta_{\Omega}(|x|),
\end{equation}
where $f:\R^n\to \R_+$ has Lipschitz gradient whose Lipschitz constant is $L_f > 0$, the absolute value is taken componentwise, $\lambda> 0$, $\Psi(y) := \sum_{i=1}^n\psi(y_i)$, $\delta_\Omega$ is the indicator function of the set $\Omega$ (see Section~\ref{sec2} for notation), and the function $\psi$ and the set $\Omega$ satisfy the following assumption:
\begin{assumption}\label{Assump1}
  \begin{enumerate}[{\rm (a)}]
    \item $\psi:\R_+\to \R_+$ satisfies $\psi(0) = 0$ and is continuous and concave. Moreover, $\psi$ is differentiable on $(0,\infty)$ with $\psi' > 0$ and $\lim_{t\to\infty}\psi(t) = \infty$.
    \item $\phi:= \psi^{-1}$ exists on $\R_+$. Moreover, $\phi'_+$ is locally Lipschitz continuous on $\R_+$.\footnote{Here, $\phi'_+(t)$ is the right-hand derivative at $t$ defined as $\lim_{h\downarrow 0}\frac{\phi(t+h) - \phi(t)}{h}$.}
    \item $\Omega\subseteq \R^n_+$ is a nonempty closed set such that a projection onto $\psi(\Omega)$ can be computed efficiently.\footnote{See Section~\ref{sec2} for the definition of $\psi(\Omega)$, the discussion on its closedness and nonemptiness, and the observation that the set of projections from any $x\in \R^n$ onto $\psi(\Omega)$ is nonempty.}
  \end{enumerate}
\end{assumption}
\noindent The above assumption is general enough for \eqref{Prob} to cover some important instances of \eqref{Prob0} and \eqref{Prob1} as special cases. For example, model \eqref{Prob0} with $\theta(t)= t$, $t^p$ ($p \in (0,0.5]$) or $\log(1+t/\epsilon)$ ($\epsilon > 0$) corresponds to \eqref{Prob} with $\Omega = \R^n_+$ and $\psi(t) = \theta(t)$; moreover, one can check that Assumption~\ref{Assump1} is satisfied for these $\psi$ and $\Omega$.\footnote{\color{blue} Note that $\psi(t) = t^p$ with $p \in (0.5,1)$ does not satisfy Assumption~\ref{Assump1} (b).} In addition, one can also cover model \eqref{Prob1} (where $\theta(t)=t$) and its variants with nonconvex regularizer $\theta(t) = t^p$ ($p \in (0,0.5]$) or $\theta(t) = \log(1+t/\epsilon)$ by setting
\[
\Omega = \{x\in \R^n_+:\; x_1\ge x_2\ge \cdots \ge x_n\},
\]
and considering $\psi(t) = \theta(t)$; indeed, recalling that projections onto the above $\Omega$ can be computed efficiently via PAVA and noting that $\psi(\Omega) = \Omega$ in this case, we see that Assumption~\ref{Assump1} is satisfied for these $\psi$ and $\Omega$.

Since \eqref{Prob} under Assumption~\ref{Assump1} is quite general and covers some particular instances of \eqref{Prob1}, it is not immediately clear how it can be solved efficiently. For example, in view of the smoothness of $f$, it might be tempting to apply proximal-gradient-type methods. However, it is unclear whether such methods can be efficient because the proximal mapping of the nonsmooth part $\lambda\Psi(|\cdot|) + \delta_{\Omega}(|\cdot|)$ in \eqref{Prob} is in general difficult to compute. Specifically, under Assumption~\ref{Assump1}, it is not necessarily easy to obtain an element of the following set given $\gamma > 0$ and $y\in \R^n$:
\begin{equation}\label{hahasub}
\Argmin_{|x|\in \Omega}\left\{\frac1{2\gamma}\|x - y\|^2 + \lambda \Psi(|x|)\right\}.
\end{equation}

In this paper, by exploiting the (absolute-value) symmetry in the nonsmooth part $\lambda\Psi(|\cdot|) + \delta_{\Omega}(|\cdot|)$ in \eqref{Prob} and the \emph{invertibility} and smoothness of $\psi$, we propose a new algorithm, which we call the \underline{D}oubly \underline{M}ajorized \underline{A}lgorithm (DMA), for solving \eqref{Prob} under Assumption~\ref{Assump1}. The DMA obviates the use of \eqref{hahasub} and, instead, makes use of projections onto $\psi(\Omega)$ in each iteration. Hence, in view of {\color{blue}Assumption~\ref{Assump1} (c)}, each iteration of DMA can be executed efficiently; moreover, one can show that the sequence $\{x^k\}$ generated satisfies $|x^k|\in \Omega$ for all $k$. To study the limiting behavior of this sequence, we define a new notion of stationarity (called $\psi_{\rm opt}$-stationarity) as inspired by the notion of $L$-stationarity (see \cite[Definition~2.3]{BeckEldar13} and \cite[Definition~5.2]{BeckHallak16}). We show that $\psi_{\rm opt}$-stationarity is a necessary condition for global optimality for \eqref{Prob} under Assumption~\ref{Assump1} {\em without} any additional assumptions, and prove that any accumulation point of the sequence generated by DMA for solving \eqref{Prob} under Assumption~\ref{Assump1} is a $\psi_{\rm opt}$-stationary point. We would like to emphasize that this characterization of accumulation points is obtained {\em without} invoking any commonly used constraint qualifications in the literature for nonsmooth nonconvex problems (such as those involving the normal cones and horizon subdifferentials; see, for example, \eqref{horizon} below); this is an advantage because such constraint qualifications can be difficult to verify in view of the complexity of the nonsmooth part in \eqref{Prob}.

The rest of the paper is organized as follows. We present notation and preliminary materials in Section~\ref{sec2}. The notion of $\psi_{\rm opt}$-stationarity is defined and shown to be necessary for global optimality in Section~\ref{sec3}. In Section~\ref{sec4}, we describe our algorithm and establish its convergence. Finally, numerical experiments on order-constrained compressed sensing problems and block order-constrained sparse time-lagged regression problems are conducted in Section~\ref{sec5} to illustrate the effectiveness of our algorithm for solving them.

\section{Notation and preliminaries}\label{sec2}

In this paper, we use $\R^n$ to denote the $n$-dimensional Euclidean space and $\R^n_+$ (\emph{resp.} $\R_{++}^n$) to denote the nonnegative (\emph{resp.} positive) orthant of $\R^n$. For two vectors $x,\,y\in \R^n$, their standard inner product is denoted by $\langle x,y\rangle$ and their Hadamard (entry-wise) product is denoted by $x\circ y$. For a vector $x\in \R^n$, we use $\|x\|$ to denote its Euclidean norm, i.e., $\|x\| = \sqrt{\langle x,x\rangle}$. We let $|x|$ be the vector whose $i$-th entry equals $|x_i|$, and ${\rm sgn}(x)$ be the vector whose $i$th entry is given by
\[
\big({\rm sgn}(x)\big)_i = \begin{cases}
  1 & {\rm if}\ x_i \ge 0,\\
  -1 & {\rm if}\ x_i < 0.
\end{cases}
\]
For a $y \in \R^n_+$ and the $\psi$ and $\phi$ as in {\color{blue}Assumptions~\ref{Assump1} (a)} and (b), with an abuse of notation, we use $\psi(y)$ and $\phi(y)$ to denote the vectors whose $i$-th entries are $\psi(y_i)$ and $\phi(y_i)$, respectively.
The vector of all ones is denoted by $e$, whose dimension should be clear from the context.

An extended real-valued function $f:\R^n\to (-\infty,\infty]$ is said to be proper if ${\rm dom}\,f := \{x\in \R^n:\; f(x) < \infty\}$ is nonempty. Such a function is said to be closed if it is lower semicontinuous. For a proper closed function $f$, the Frech\'et subdifferential $\widehat \partial f$, the (limiting) subdifferential $\partial f$ and the horizon subdifferential $\partial^\infty f$ of $f$ at an $x\in {\rm dom}\,f$ are defined respectively as
\[
\begin{aligned}
  \widehat \partial f(x) &:= \left\{v\in \R^n:\; \liminf_{u\to x, u\neq x}\frac{f(u) - f(x) - \langle v,u-x\rangle}{\|u - x\|}\ge 0\right\};\\
  \partial f(x) &:= \{v\in \R^n:\; \exists v^k \to v, x^k\overset{f}\to x \ {\rm with}\ v^k\in \widehat\partial f(x^k)\ \mbox{for all }k\};\\
  \partial^\infty f(x) &:= \{v\in \R^n:\; \exists  t_k\downarrow 0, v^k \to v, x^k\overset{f}\to x \ {\rm with}\ v^k/t_k\in \widehat\partial f(x^k)\ \mbox{for all }k\};
\end{aligned}
\]
where $x^k \overset{f}\to x$ means $x^k \to x$ and $f(x^k)\to f(x)$. Moreover, for $x\notin {\rm dom}\,f$, one defines $\widehat\partial f(x) = \partial f(x) = \partial^\infty f(x) = \emptyset$. As we will comment in Section~\ref{sec3} below, these subdifferentials are standard tools for deriving optimality conditions for \eqref{Prob}, and we defer the discussion on subdifferential-based optimality conditions to Section~\ref{sec3.1}.
For a nonempty closed set $\Xi$ in $\R^n$, we let $P_{\Xi}$ denote the set of projections onto it, i.e., for all $x\in \R^n$,
\[
P_{\Xi}(x) := \Argmin_{y\in \Xi}\|x - y\|.
\]
The above set reduces to a singleton if $\Xi$ is in addition convex.
The indicator function of $\Xi$ is denoted by $\delta_{\Xi}$, which equals $0$ if $x\in \Xi$ and equals $\infty$ otherwise. Moreover, the normal cone of $\Xi$ at an $x\in \Xi$ is defined as $N_{\Xi}(x):= \partial \delta_{\Xi}(x)$. For a nonempty closed set $\Theta\subseteq \R^n_+$ and the $\psi$ as in {\color{blue}Assumption~\ref{Assump1} (a)}, we write
\[
\psi(\Theta) := \{\psi(y)\in \R^n:\; y\in \Theta\} \subseteq \R^n_+,
\]
where we abuse the notation $\psi(y)$ as described above, i.e., use $\psi(y)$ to denote the vector whose $i$-th entry is $\psi(y_i)$.
Note that this set is necessarily closed and nonempty under {\color{blue}Assumptions~\ref{Assump1} (a)} and (b) because $\psi:\R_+\to \R_+$ is a continuous bijection on $\R_+$. Hence, $P_{\psi(\Theta)}(x)\neq \emptyset$ for all $x\in \R^n$.

In the remainder of this section, we will present several important auxiliary lemmas. We start with the following version of Taylor's inequality for (finite-valued) continuous convex functions on $\R_+$. Here and throughout, for a function $h:\R_+\to \R$, we let $h'_+$ denote its right-hand derivative, i.e.,
\[
h'_+(t) := \lim_{s\downarrow 0}\frac{h(t+s) - h(t)}{s}\ \ \ \forall t\in \R_+.
\]
\begin{lemma}\label{lem:Taylor0}
 Let $h:\R_+\to \R$ be convex and continuous, and $h$ be differentiable on $(0,\infty)$. Suppose also that $h'_+$ is locally Lipschitz continuous on $\R_+$. Then for any $a >0$, there exists $c > 0$ such that
 \[
 h(t) + h'_+(t)(s-t)\le h(s) \le h(t) + h'_+(t)(s-t) + \frac{c}{2}(s-t)^2\ \  {\color{blue}whenever}\ s,t \in [0,a].
 \]
\end{lemma}
\begin{proof}
  The first inequality is a direct consequence of convexity.
  For the second inequality, define the function $H:\R\to \R$ by
  \[
  H(t) := \begin{cases}
    h(t) & {\rm if }\ t \ge 0,\\
   h(0) + h'_+(0) t & {\rm otherwise}.
  \end{cases}
  \]
  Then by direct computation and the continuity of $h'_+$ on $\R_+$, we have
  \[
  H'(t) := \begin{cases}
    h'(t) & {\rm if }\ t > 0,\\
    h'_+(0) & {\rm otherwise}.
  \end{cases}
  \]
  Since $h'_+$ is locally Lipschitz continuous on $\R_+$, we see from the above formula that $H'$ is locally Lipschitz on $\R$.
  The desired inequality (and the existence of $c > 0$) now follows from the standard descent lemma for functions with locally Lipschitz gradients.
\end{proof}

The next lemma concerns a version of Taylor's inequality for a structured function defined on $\R_+$. We will make use of the explicit formula of $L$ in \eqref{def_L} for our convergence analysis in Section~\ref{sec4}.

\begin{lemma}[A descent lemma]\label{lem:Taylor}
 Let $h:\R_+\to \R$ be convex and continuous, and $h$ be differentiable on $(0,\infty)$. Suppose also that $h'_+$ is locally Lipschitz continuous on $\R_+$. Let $b \in \R$, $\gamma > 0$ and define $g(t) = \frac1{2\gamma}(h(t) - b)^2$. Then for any $a >0$, it holds that
 \[
 g(s) \le g(t) + g'_+(t)(s-t) + \frac{L}{2}(s-t)^2\ \ {\color{blue} whenever}\ s,t \in [0,a],
 \]
 where
 \begin{equation}\label{def_L}
 L = \frac{c}{\gamma}\left(\sup_{t\in [0,a]}|h(t)| + |b|\right) + \frac1{\gamma}\left(\sup_{t\in [0,a]}|h'_+(t)| + \frac{ac}2\right)^2 < \infty,
 \end{equation}
 with $c > 0$ given in Lemma~\ref{lem:Taylor0}.
\end{lemma}
\begin{proof}
  By direct computation, we see that
  \[
  g'_+(t) = \frac1\gamma (h(t) - b) h'_+(t)\ \ \ {\rm whenever \ }t\in \R_+.
  \]
  Now, fix any $a>0$. From Lemma~\ref{lem:Taylor0}, we can find $c>0$ so that
  \begin{equation}\label{boundhaha}
  |h(s) - h(t) - h'_+(t)(s-t)| \le \frac{c}{2}(s-t)^2\ \ {\rm whenever \ }s,t\in [0,a].
  \end{equation}
  Consequently, for any $s,\,t\in [0,a]$, we have
  \begin{equation}\label{2gammags}
  \begin{aligned}
   2\gamma g(s)& = (h(s) - b)^2 = [h(t) - b + h'_+(t)(s-t) + h(s) - h(t) - h'_+(t)(s-t)]^2\\
   & = (h(t)-b)^2 + 2(h(t)-b)h'_+(t)(s-t)\\
   & \ \ \ + 2(h(t)-b)[h(s) - h(t) - h'_+(t)(s-t)]\\
   & \ \ \ + (h'_+(t)(s-t) + h(s) - h(t) - h'_+(t)(s-t))^2.
  \end{aligned}
  \end{equation}
  We now drive upper bounds for the third and fourth terms on the right hand side of \eqref{2gammags}. For the third term, observe from \eqref{boundhaha} that
  \begin{equation}\label{relation1}
  \begin{aligned}
  &2(h(t)-b)[h(s) - h(t) - h'_+(t)(s-t)]\\
  &\le c|h(t) - b|(s-t)^2\le c\left(\sup_{\hat t\in [0,a]}|h(\hat t)| + |b|\right)(s-t)^2.
  \end{aligned}
  \end{equation}
  Next, for the fourth term on the right hand side of \eqref{2gammags}, we can also deduce using \eqref{boundhaha} that
  \begin{equation}\label{relation2}
    \begin{aligned}
      &(h'_+(t)(s-t) + h(s) - h(t) - h'_+(t)(s-t))^2\\
      & \le (|h'_+(t)(s-t)| + |h(s) - h(t) - h'_+(t)(s-t)|)^2\\
      & \le \left(\sup_{\hat t\in [0,a]}|h'_+(\hat t)||s-t| + \frac{c}2(s-t)^2\right)^2 = \left(\sup_{\hat t\in [0,a]}|h'_+(\hat t)| + \frac{c}2|s-t|\right)^2(s-t)^2\\
      & \le \left(\sup_{\hat t\in [0,a]}|h'_+(\hat t)| + \frac{ca}2\right)^2(s-t)^2,
    \end{aligned}
  \end{equation}
  where the last inequality holds because $s,t\in [0,a]$. Combining \eqref{relation1} and \eqref{relation2} with \eqref{2gammags} and invoking the definition of $L$ in \eqref{def_L}, we can now obtain
  \begin{align*}
   2\gamma g(s) & \le (h(t)-b)^2 + 2(h(t)-b)h'_+(t)(s-t) +\gamma L(s-t)^2\\
   & = 2\gamma g(t) + 2\gamma g'_+(t)(s-t) + \gamma L(s-t)^2,
  \end{align*}
  where the equality follows from the definition of $g$ and the formula of $g'_+$. This completes the proof.
\end{proof}

Before ending this section, we present a key lemma concerning properties of the optimal solution of an absolutely symmetrically structured problem.
\begin{lemma}[Minimizers under absolute-value symmetry]\label{lem:key}
Let $\Upsilon\subseteq\R_+^n$ be a nonempty closed set, $\widebar{x}\in\R^n$ and $g: \R_+^n\to\R_+$ be lower semicontinuous. If
\begin{equation}\label{subp_form}
u\in\Argmin_{|x|\in\Upsilon}\Big\{\frac{1}{2}\|x - \widebar{x}\|^2 + g(|x|)\Big\},
\end{equation}
then we have
\begin{equation}\label{key_obs}
|u| \in\Argmin_{w\in\Upsilon}\Big\{\frac{1}{2}\|w - |\widebar{x}|\|^2 + g(w)\Big\}.
\end{equation}
Moreover, if $\Argmin_{|x|\in\Upsilon}\Big\{\frac{1}{2}\|x - \widebar{x}\|^2 + g(|x|)\Big\} = \{u\}$, then we have $\sgn(u_i) = \sgn(\widebar{x}_i)$ whenever $u_i\neq 0$.
\end{lemma}

\begin{proof}
We first prove the following relationship.
\begin{equation}\label{eqs_eq}
{\rm val}_1:= \inf_{|x|\in\Upsilon}\Big\{\frac{1}{2}\|x - \widebar{x}\|^2 + g(|x|)\Big\} = \inf_{|x|\in\Upsilon}\Big\{\frac{1}{2}\||x| - |\widebar{x}|\|^2 + g(|x|)\Big\} =:{\rm val}_2.
\end{equation}
Due to the triangle inequality $\||x| - |\widebar{x}|\| \le \|x - \widebar{x}\|$, it suffices to show that ${\rm val}_1 \le {\rm val}_2$.
To this end, let $x^*$ be a solution of the problem on the right-hand side in \eqref{eqs_eq}.\footnote{Note that such a solution exists because $\Upsilon$ is closed and nonempty, and $g$ is nonnegative lower semicontinuous.} Define $\widehat{x}^*\in \R^n$ by
\begin{equation*}
\widehat{x}^*_i =
\begin{cases}
|x^*_i| & {\rm if}\ \widebar{x}_i = 0;\\
|x^*_i|\cdot\sgn(\widebar{x}_i) & {\rm else}.
\end{cases}
\end{equation*}
Then we have $|\widehat{x}^*| = |x^*|\in\Upsilon$. Moreover,
\begin{equation*}
 {\rm val}_1 \le \frac{1}{2}\|\widehat{x}^* - \widebar{x}\|^2 + g(|\widehat{x}^*|)
  = \frac{1}{2}\||x^*| - |\widebar{x}|\|^2 + g(|x^*|) = {\rm val}_2.
\end{equation*}
This proves \eqref{eqs_eq}.

Next, {\color{blue}we prove the following two statements concerning  $u$ satisfying \eqref{subp_form}:}
\begin{itemize}
    \item[(i)] when $\widebar{x}_i > 0$, it holds that $u_i \ge 0$; when $\widebar{x}_i < 0$, it holds that $u_i \le 0$;
    \item[(ii)] when $\widebar{x}_i = 0$, if $u_i\neq 0$, then $u$ is not the unique solution in \eqref{subp_form}.
\end{itemize}
For (i), if there exists some $i$ such that $\widebar{x}_i > 0$ but $u_i < 0$ or $\widebar{x}_i < 0$ but $u_i > 0$, we can pick any such $i$ and define $\widebar{u}\in \R^n$ by
\begin{equation}\label{u_bar}
    \widebar{u}_j =
    \begin{cases}
    u_j & \ {\rm if}\ j \neq i,\\
    -u_i & \ {\rm if}\ j = i.
    \end{cases}
\end{equation}
Then one can see that $|\widebar{u}| = |u|\in\Upsilon$ and
\begin{equation*}
    \frac{1}{2}\|\widebar{u} - \widebar{x}\|^2 + g(|\widebar{u}|) <  \frac{1}{2}\|u - \widebar{x}\|^2 + g(|u|),
\end{equation*}
which contradicts \eqref{subp_form}.

For (ii), if there exists some $i$ such that $\widebar{x}_i = 0$ and $u_i\neq 0$, we can still define $\widebar{u}$ as in \eqref{u_bar}. Since $u_i\neq 0$, we have $\widebar{u} \neq u$. Moreover, $\widebar{u}$ is also a solution in \eqref{subp_form}, thanks to $\widebar{x}_i = 0$. Therefore, $u$ is not the unique solution in \eqref{subp_form}.

Now, we are ready to prove \eqref{key_obs}. We have
\begin{equation*}
\begin{split}
 \frac{1}{2}\||u| - |\widebar{x}|\|^2 + g(|u|) & \overset{\rm (a)}{=} \frac{1}{2}\|u - \widebar{x}\|^2 + g(|u|) = \inf_{|x|\in\Upsilon}\Big\{\frac{1}{2}\|x - \widebar{x}\|^2 + g(|x|)\Big\}\\
 & \overset{\rm (b)}{=} \inf_{|x|\in\Upsilon}\Big\{\frac{1}{2}\||x| - |\widebar{x}|\|^2 + g(|x|)\Big\} \overset{\rm (c)}{=} \inf_{w\in\Upsilon}\Big\{\frac{1}{2}\|w - |\widebar{x}|\|^2 + g(w)\Big\}{\color{blue},}
\end{split}
\end{equation*}
where (a) follows from statement (i), (b) follows from \eqref{eqs_eq} and (c) follows from the fact that $\Upsilon\subseteq\R_+^n$. This together with $|u|\in\Upsilon$ proves \eqref{key_obs}.

Finally, if $u$ is the unique solution in \eqref{subp_form}, we see from statement (ii) that if $\widebar{x}_i = 0$ we will have $u_i = 0$. Consequently, we have $\widebar{x}_i\neq 0$ whenever $u_i\neq 0$. Then statement (i) implies that $\sgn(u_i) = \sgn(\widebar{x}_i)$ whenever $u_i\neq 0$. This completes the proof.
\end{proof}

\begin{remark}\label{rem:converse}
  Let $\Upsilon\subseteq\R_+^n$ be a nonempty closed set, $\widebar{x}\in\R^n$ and $g: \R_+^n\to\R_+$ be lower semicontinuous. Then one can observe from the proof of Lemma~\ref{lem:key} that \eqref{eqs_eq} holds. Using this observation, one can show readily that if
  $\breve w \in \Argmin_{w\in\Upsilon}\Big\{\frac{1}{2}\|w - |\widebar{x}|\|^2 + g(w)\Big\}$, then $\breve x:= {\rm sgn}(\widebar{x})\circ \breve w$ belongs to $\Argmin_{|x|\in\Upsilon}\Big\{\frac{1}{2}\|x - \widebar{x}\|^2 + g(|x|)\Big\}$. We will need this fact when developing our algorithm in Section~\ref{sec4}.
\end{remark}

\section{First-order necessary optimality conditions}\label{sec3}

In this section, we discuss (first-order) necessary optimality conditions for \eqref{Prob} under Assumption~\ref{Assump1}. Specifically, we discuss necessary conditions for a feasible point of \eqref{Prob} to be globally optimal.

\subsection{Necessary optimality conditions based on limiting subdifferential}\label{sec3.1}

One large class of necessary optimality conditions is deduced based on the concept of limiting subdifferential; see \cite{LiSoMa20} for a recent overview. Indeed, in view of \cite[Theorem~10.1]{RockWets98} and \cite[Exercise~8.8(c)]{RockWets98}, we know that if $x^*$ is a global minimizer of \eqref{Prob} under Assumption~\ref{Assump1}, then
\begin{equation}\label{stat}
0 \in \partial F(x^*) =  \nabla f(x^*) + \partial(\lambda \Psi(|\cdot|) + \delta_{\Omega}(|\cdot|))(x^*).
\end{equation}
A point $\bar x$ satisfying \eqref{stat} in place of $x^*$ is called a stationary point of the function $F$ in \eqref{Prob}. Notice that the subdifferential of the nonsmooth part $P(x) := \lambda\Psi(|x|) + \delta_\Omega(|x|)$, however, is in general difficult to characterize. In addition, typical algorithm such as the proximal gradient algorithm, which clusters at such stationary points, needs to compute the proximal mapping of $P$ in each iteration, i.e., to find
\[
\breve x\in \Argmin_{|x|\in \Omega} \frac1{2\gamma}\|x - y\|^2 + \lambda \Psi(|x|)
\]
given $y$ and $\gamma> 0$.
To the best of our knowledge, such an $\breve x$ cannot be found efficiently for general $\Omega$ and $\Psi$ satisfying Assumption~\ref{Assump1}.

{\em Simpler} subdifferential-based necessary optimality conditions can be obtained under suitable {\em constraint qualifications} at $x^*$ involving the {\em horizon} subdifferential, such as
\begin{equation}\label{horizon}
    \partial^\infty\Psi(|\cdot|)(x^*)\cap {\color{blue}\left(-N_{\widehat\Omega}(x^*)\right)} = \{0\},
\end{equation}
where $\widehat\Omega := \{x:\; |x|\in \Omega\}$. Indeed, under \eqref{horizon}, it is possible to deduce using \eqref{stat} and \cite[Corollary~10.9]{RockWets98} that if $x^*$ is a global minimizer of \eqref{Prob} under Assumption~\ref{Assump1}, then
\begin{equation*}
    0 \in \nabla f(x^*) + \lambda \partial \Psi(|\cdot|)(x^*) + N_{\widehat\Omega}(x^*).
\end{equation*}
While the one single set of subdifferential in \eqref{stat} is split into two in the above display, with each subdifferential set considerably easier to characterize, the sum of the two sets is still not easy to characterize. Furthermore, it also appears to be nontrivial to verify \eqref{horizon} at a candidate solution $x^*$ for our particular problem \eqref{Prob}; indeed, it is unclear whether such condition should hold at any global minimizer of our problem~\eqref{Prob}.

In view of the complicated structure of \eqref{Prob} and the aforementioned difficulties, in this paper, we focus on another way of deriving necessary optimality conditions for our problem \eqref{Prob}. This alternative approach does not explicitly involve subdifferentials and is constraint-qualification free. It is based on fixed {\color{blue}points} of set-valued maps.

\subsection{Necessary optimality conditions based on fixed points of set-valued maps}

Suppose that $x^*$ is a global minimizer of \eqref{Prob} under Assumption~\ref{Assump1}. Then, using the Lipschitz continuity of $\nabla f$, we see that for any $x$ satisfying $|x|\in \Omega$,
\[
\begin{aligned}
  F(x^*)&\le F(x) = f(x) + \lambda\Psi(|x|)\\
  &\le f(x^*) + \langle \nabla f(x^*),x-x^*\rangle + \frac{L_f}{2}\|x - x^*\|^2 + \lambda \Psi(|x|).
\end{aligned}
\]
This shows that
\begin{equation*}
    x^* \in \Argmin_{|x|\in \Omega} \left\{\langle \nabla f(x^*),x-x^*\rangle + \frac{L_f}{2}\|x - x^*\|^2 + \lambda \Psi(|x|)\right\}.
\end{equation*}
Motivated by the above argument, the notion of $L$-stationarity (see \cite[Definition~2.3]{BeckEldar13} and \cite[Definition~5.2]{BeckHallak16}) and the proof of \cite[Lemma~5.3]{BeckHallak16}, we say that an $x^*$ is an $L$-stationary point of \eqref{Prob} (under Assumption~\ref{Assump1}) if there exists $\eta > 0$ such that
\begin{equation}\label{Lstationary}
    x^* \in \Argmin_{|x|\in \Omega} \left\{\langle \nabla f(x^*),x-x^*\rangle + \frac{\eta}{2}\|x - x^*\|^2 + \lambda \Psi(|x|)\right\}.
\end{equation}
Note that if we define a set-valued map $S_\eta:\R^n\rightrightarrows \R^n$ by
\[
S_\eta(y) := \Argmin_{|x|\in \Omega} \left\{\langle \nabla f(y),x-y\rangle + \frac{\eta}{2}\|x - y\|^2 + \lambda \Psi(|x|)\right\},
\]
then $x^*$ is an $L$-stationary point of \eqref{Prob} if and only if there exists $\eta > 0$ such that $x^*$ is a fixed point of $S_\eta$, i.e., $x^*\in S_\eta(x^*)$.

In view of \cite[Theorem~10.1]{RockWets98} and \cite[Exercise~8.8(c)]{RockWets98}, we can deduce immediately that if $x^*$ is an $L$-stationary point of \eqref{Prob}, then $x^*$ is a stationary point of $F$, i.e., it satisfies \eqref{stat}. Moreover, one can show that if the proximal gradient algorithm is applied to solving \eqref{Prob}, any accumulation point is an $L$-stationary point; in this sense, we can regard the proximal gradient algorithm as a {\em companion} algorithm for the notion of $L$-stationarity. However, as pointed out in the previous subsection, it is not clear whether the proximal gradient algorithm can be applied efficiently to solve \eqref{Prob}.

In this paper, we further relax the notion of optimality in the fixed point inclusion in \eqref{Lstationary} and define the following notion of $\psi_{\rm opt}$-stationarity for \eqref{Prob} under Assumption~\ref{Assump1}. The name $\psi_{\rm opt}$ suggests that this notion of stationarity involves an optimization problem concerning $\psi$.

\begin{definition}[$\psi_{\rm opt}$-stationarity condition]
\label{def_ws}
Consider \eqref{Prob} and suppose that Assumption~\ref{Assump1} holds. We say that $x^*$ satisfies the $\psi_{\rm opt}$-stationarity condition, if there exist some $\eta_* > 0$ and $\alpha^*\in \{-1,1\}^n$ such that
\begin{equation}\label{x_opt}
      x^* \in \Argmin_{|x|\in \Omega}\left\{\frac{\eta_*}2\|\psi(|x|) - \psi(|x^*|)\|^2 + \sum_{i=1}^n\big[\lambda + \alpha^*_i\cdot\nabla_i f(x^*) \cdot\phi_+'(\psi(|x^*_i|))\big]\cdot \psi(|x_i|)\right\},
  \end{equation}
with $\alpha_i^* = {\rm sgn}(x_i^*)$ whenever $x_i^* \neq 0$, and $\alpha_i^* = -{\rm sgn}(\nabla_i f(x^*))$ if $x_i^*=0$ and $\nabla_i f(x^*)\neq 0$.
\end{definition}
We show in the next theorem that the $\psi_{\rm opt}$-stationarity condition is a necessary condition for global optimality of \eqref{Prob} under Assumption~\ref{Assump1}{\color{blue}; the proof involves two majorization steps, which arise in \eqref{L_prop} and \eqref{part_2} below.} Interestingly, this implication does not require any additional assumptions such as \eqref{horizon}. We will develop a {\em companion algorithm} for this notion of stationarity in Section~\ref{sec4}.

\begin{theorem}[Global optimality implies $\psi_{\rm opt}$-stationarity]\label{thm:opt-to-psiopt}
Consider \eqref{Prob} and suppose that Assumption~\ref{Assump1} holds. Then any global minimizer of \eqref{Prob} satisfies the $\psi_{\rm opt}$-stationarity condition.
\end{theorem}
\begin{proof}
Let $x^*$ be a global minimizer of \eqref{Prob}. We then have $|x^*|\in\Omega$ and for any $|x|\in\Omega$,
\begin{equation*}
    f(x^*) + \lambda\Psi(|x^*|) \le f(x) + \lambda\Psi(|x|).
\end{equation*}
Since $f$ has Lipschitz gradient with modulus $L_f$, we have for any $|x|\in\Omega$,
\begin{equation}\label{L_prop}
    \lambda\Psi(|x^*|) \le  f(x) - f(x^*) + \lambda\Psi(|x|) \le  \langle \nabla f(x^*),\, x - x^*\rangle + \frac{L_f}{2}\|x - x^*\|^2 + \lambda\Psi(|x|).
\end{equation}
Fix any $\widebar{L}_f > L_f$. Then \eqref{L_prop} further implies that
\begin{equation}\label{start_gbm}
\begin{split}
x^* & \in \Argmin_{|x|\in\Omega} \Big\{\langle \nabla f(x^*),\, x - x^*\rangle + \frac{\widebar{L}_f}{2}\|x - x^*\|^2 + \lambda\Psi(|x|)\Big\}\\
& = \Argmin_{|x|\in\Omega}\Big\{\frac{\widebar{L}_f}{2}\big\|x - (x^* - \nabla f(x^*)/\widebar{L}_f)\big\|^2 + \lambda\Psi(|x|)\Big\}.
\end{split}
\end{equation}
Now we show that $x^*$ is the unique optimal solution in \eqref{start_gbm}. Suppose to the contrary that there exists another optimal solution $\widebar{x}^*\neq x^*$. We then see from this and $\widebar{L}_f > L_f$ that
\begin{equation*}
\begin{split}
 & \langle \nabla f(x^*),\, \widebar{x}^* - x^*\rangle + \frac{L_f}{2}\|\widebar{x}^* - x^*\|^2 + \lambda\Psi(|\widebar{x}^*|)\\
 & < \langle \nabla f(x^*),\, \widebar{x}^* - x^*\rangle + \frac{\widebar{L}_f}{2}\|\widebar{x}^* - x^*\|^2 + \lambda\Psi(|\widebar{x}^*|)\\
 & \overset{\rm (a)}= \langle \nabla f(x^*),\, x^* - x^*\rangle + \frac{\widebar{L}_f}{2}\|x^* - x^*\|^2 + \lambda\Psi(|x^*|)= \lambda\Psi(|x^*|),
  \end{split}
\end{equation*}
where we used the assumption that both $\bar x^*$ and $x^*$ are optimal solutions in (a). The above display contradicts \eqref{L_prop} and hence $x^*$ is the unique solution in \eqref{start_gbm}.

By Lemma~\ref{lem:key}, we have from \eqref{start_gbm} that
\begin{equation}\label{two_eq}
\begin{split}
\sgn(x_i^*) & = \sgn\left(x_i^* - \nabla_i f(x^*)/\widebar{L}_f\right) \ \ \ {\rm when}
 \ \ x_i^*\neq 0, \\
|x^*|&\in\Argmin_{w\in\Omega}\Big\{\frac{\widebar{L}_f}{2}\big\|w - |x^* - \nabla f(x^*)/\widebar{L}_f|\big\|^2 + \lambda\Psi(w)\Big\}.
\end{split}
\end{equation}
Let $v = \psi(w)$. {\color{blue}Recalling $\phi = \psi^{-1}$, we} then have
\begin{equation}\label{psi_arg}
    \psi(|x^*|)\in\Argmin_{v\in\psi(\Omega)}\Big\{\frac{\widebar{L}_f}{2}\big\|\phi(v)- |x^* - \nabla f(x^*)/\widebar{L}_f|\big\|^2 + \lambda \langle e,v\rangle\Big\}.
\end{equation}
For each $i$, let $g_i(v_i):= \frac{\widebar{L}_f}{2}\big(\phi(v_i) - |x_i^* - \nabla_if(x^*)/\widebar{L}_f|\big)^2$ and
\begin{equation}\label{def_alpha_s}
    \alpha_i^* :=
    \begin{cases}
    -\sgn(\nabla_if(x^*)) & {\rm if}\ x_i^* = 0\ {\rm and}\ \nabla_if(x^*)\neq 0;\\
     \sgn(x_i^*) & {\rm else.}
    \end{cases}
\end{equation}
We then obtain from the local Lipschitz continuity of $\phi_+'$ that
\begin{equation}\label{g_prim}
\begin{split}
    (g_i)_+'(\psi(|x_i^*|))& = \widebar{L}_f(\phi(\psi(|x_i^*|)) - |x_i^* - \nabla_i f(x^*)/\widebar{L}_f|)\cdot\phi_+'(\psi(|x_i^*|))\\
    & = \widebar{L}_f(|x_i^*| - |x_i^* - \nabla_i f(x^*)/\widebar{L}_f|)\cdot\phi_+'(\psi(|x_i^*|))\\
    & \overset{\rm (a)}= \widebar{L}_f\,\alpha_i^*(x_i^* - x_i^* + \nabla_i f(x^*)/\widebar{L}_f)\cdot\phi_+'(\psi(|x_i^*|))\\
    & = \alpha_i^*\cdot\nabla_if(x^*)\cdot\phi_+'(\psi(|x_i^*|)),
\end{split}
\end{equation}
where (a) follows from the definition of $\alpha_i^*$ in \eqref{def_alpha_s} and the first equation in \eqref{two_eq}.

On the other hand, upon rewriting \eqref{psi_arg} as $\psi(|x^*|) \in\Argmin_{v\in\psi(\Omega)}\left\{\sum_{i=1}^ng_i(v_i) + \lambda \langle e,v\rangle\right\}$, we see that for any $v\in\psi(\Omega)$,
\begin{equation}\label{part_1}
\begin{aligned}
    0 &\le \sum_{i=1}^ng_i(v_i) - \sum_{i=1}^ng_i(\psi(|x_i^*|)) + \lambda \langle e,v - \psi(|x^*|)\rangle\\
    &= \sum_{i=1}^n\left[g_i(v_i) - g_i(\psi(|x_i^*|)) + \lambda(v_i - \psi(|x_i^*|))\right].
\end{aligned}
\end{equation}

Now, fix any $R > 0$ and define $B_R:=\{v: \|v - \psi(|x^*|)\|\le R\}$. For all $v\in\psi(\Omega)\cap B_R$, we have $|v_i|\le \|\psi(|x^*|)\| + R$ for all $i$. We then see from Lemma~\ref{lem:Taylor}\footnote{Notice that $\phi:\R_+\to \R_+$ is convex because $\psi$ is concave and monotone. Moreover, the differentiability of $\phi$  on $(0,\,\infty)$ follows from $\psi' >0$ on $(0,\,\infty)$.} that there exists $\eta_R > 0$ such that for all $v\in\psi(\Omega)\cap B_R$ and all $i$,
\begin{equation}\label{part_2}
    g_i(v_i) - g_i(\psi(|x_i^*|)) \le (g_i)_+'(\psi(|x_i^*|))(v_i - \psi(|x_i^*|)) + \frac{\eta_R}{2}(v_i - \psi(|x_i^*|))^2.
\end{equation}
For simplicity of notation, we let $\mu_i:= \lambda + (g_i)_+'(\psi(|x_i^*|))$ and $\mu:= (\mu_1,\ldots,\mu_n)^\top$. Let $\eta_*:=\max\{\eta_R,\, 2\|\mu\|/R\}$ and define
\begin{equation}\label{V_star}
V_{\eta_*}:=\Argmin_{v\in\psi(\Omega)}\Big\{\frac{\eta_*}{2}\|v - \psi(|x^*|)\|^2 + \sum_{i=1}^n\mu_i(v_i - \psi(|x_i^*|)) \Big\}.
\end{equation}
For any $v_{\eta_*}\in V_{\eta_*}$, we have from $\psi(|x^*|)\in\psi(\Omega)$ that
\begin{equation*}
\begin{aligned}
    &\frac{\eta_*}{2}\|v_{\eta_*} - \psi(|x^*|)\|^2 + \langle \mu, v_{\eta_*} - \psi(|x^*|)\rangle\\
    &\le \frac{\eta_*}{2}\|\psi(|x^*|) - \psi(|x^*|)\|^2 + \langle \mu, \psi(|x^*|) - \psi(|x^*|)\rangle = 0.
\end{aligned}
\end{equation*}
This together with the definition of $\eta_*$ and the Cauchy-Schwartz inequality further gives
\begin{equation*}
    \|v_{\eta_*} - \psi(|x^*|)\| \le \frac{2\|\mu\|}{\eta_*} \le R,
\end{equation*}
which implies that $v_{\eta_*}\in B_R$. In view of the arbitrariness of $v_{\eta_*}$, we have shown
\begin{equation}\label{part_3}
    V_{\eta_*}\subseteq B_R.
\end{equation}
Now, we combine \eqref{part_1} with \eqref{part_2}, use $\eta_*\ge\eta_R$ and obtain
\begin{equation*}
\begin{split}
\psi(|x^*|) & \in\Argmin_{v\in\psi(\Omega)\cap B_R}\Big\{\frac{\eta_*}{2}\|v - \psi(|x^*|)\|^2 + \sum_{i=1}^n\big(\lambda + (g_i)_+'(\psi(|x_i^*|))\big)(v_i - \psi(|x_i^*|)) \Big\}\\
& = \Argmin_{v\in\psi(\Omega)\cap B_R}\Big\{\frac{\eta_*}{2}\|v - \psi(|x^*|)\|^2 + \sum_{i=1}^n\mu_i(v_i - \psi(|x_i^*|)) \Big\}\\
& \overset{\rm (a)}{=} \Argmin_{v\in\psi(\Omega)}\Big\{\frac{\eta_*}{2}\|v - \psi(|x^*|)\|^2 + \sum_{i=1}^n\mu_i(v_i - \psi(|x_i^*|)) \Big\}\\
& \overset{\rm (b)}{=} \Argmin_{v\in\psi(\Omega)}\Big\{\frac{\eta_*}{2}\|v - \psi(|x^*|)\|^2 + \sum_{i=1}^n\big(\lambda + \alpha_i^*\cdot\nabla_if(x^*)\cdot\phi_+'(\psi(|x_i^*|))\big)(v_i - \psi(|x_i^*|)) \Big\},
\end{split}
\end{equation*}
where (a) follows from \eqref{V_star} and \eqref{part_3}, and (b) follows from \eqref{g_prim} and the definition of $\mu_i$.

Finally, recall that for any $|x|\in\Omega$, we have $\psi(|x|)\in\psi(\Omega)$. Then we can deduce from this and the above display that
\begin{equation*}
\begin{aligned}
    &\frac{\eta_*}{2}\|\psi(|x|) - \psi(|x^*|)\|^2 + \sum_{i=1}^n\big[\lambda + \alpha_i^*\cdot\nabla_if(x^*)\cdot\phi_+'(\psi(|x_i^*|))\big]\cdot[\psi(|x_i|) - \psi(|x_i^*|)]\\
    & \ge \frac{\eta_*}{2}\|\psi(|x^*|) - \psi(|x^*|)\|^2 + \sum_{i=1}^n\big[\lambda + \alpha_i^*\cdot\nabla_if(x^*)\cdot\phi_+'(\psi(|x_i^*|))\big]\cdot[\psi(|x^*_i|) - \psi(|x_i^*|)] = 0.
\end{aligned}
\end{equation*}
This together with the arbitrariness of $|x|\in\Omega$ and the definition of $\alpha^*$ in \eqref{def_alpha_s} shows that $x^*$ satisfies the $\psi_{\rm opt}$-stationarity condition.
\end{proof}

Next, we show that the $\psi_{\rm opt}$-stationary condition in Definition~\ref{def_ws} implies the standard notion of stationarity in \eqref{stat} when $\Omega = \R_+^n$.
\begin{proposition}[$\psi_{\rm opt}$-stationarity versus stationarity]\label{re_to_st}
Consider \eqref{Prob} and suppose that Assumption~\ref{Assump1} holds. Suppose in addition that $\Omega = \R_+^n$ and let $x^*$ satisfy the corresponding $\psi_{\rm opt}$-stationarity condition. Then the following statements hold.
\begin{enumerate}[{\rm (i)}]
    \item For all $i$ with $x_i^*\neq 0$, we have
 \begin{equation}\label{opt_i}
       0 \in \lambda \partial \psi(|\cdot|)(x_i^*) + \nabla_i f(x^*).
 \end{equation}
 \item If $\psi'_+(0)<\infty$, then we have for all $i$ with $x_i^* = 0$ that
 \begin{equation}\label{opt_i2}
       |\nabla_i f(x^*)|\le \lambda \psi'_+(0).
 \end{equation}
\end{enumerate}
\end{proposition}

\begin{proof}
By definition, there exist some $\eta_* > 0$ and $\alpha^*\in \{-1,1\}^n$ with $\alpha^*_i = \sgn(x^*_i)$ whenever $x^*_i \neq 0$, and $\alpha_i^* = -{\rm sgn}(\nabla_i f(x^*))$ if $x_i^*=0$ and $\nabla_i f(x^*)\neq 0$, such that
\begin{equation*}
      x^* \in \Argmin_{|x|\in \R^n_+}\left\{\frac{\eta_*}2\|\psi(|x|) - \psi(|x^*|)\|^2 + \sum_{i=1}^n\big[\lambda + \alpha^*_i\cdot\nabla_i f(x^*) \cdot\phi_+'(\psi(|x^*_i|))\big]\cdot \psi(|x_i|)\right\}.
  \end{equation*}
This implies that
\begin{equation*}
      |x^*| \in \Argmin_{w\in \R^n_+}\left\{\frac{\eta_*}2\|\psi(w) - \psi(|x^*|)\|^2 + \sum_{i=1}^n\big[\lambda + \alpha^*_i\cdot\nabla_i f(x^*) \cdot\phi_+'(\psi(|x^*_i|))\big]\cdot \psi(w_i)\right\}.
  \end{equation*}
Since $\psi:\R^n_+\to \R^n_+$ is invertible and $\psi(\R^n_+) = \R^n_+$, we further have
\begin{equation}\label{optimalityvv}
      \psi(|x^*|) \in \Argmin_{v\in \R^n_+}\left\{\frac{\eta_*}2\|v - \psi(|x^*|)\|^2 + \sum_{i=1}^n\big[\lambda + \alpha^*_i\cdot\nabla_i f(x^*) \cdot\phi_+'(\psi(|x^*_i|))\big]\cdot v_i\right\}.
  \end{equation}
Using the first-order optimality conditions for \eqref{optimalityvv}, we have for all $i$ with $x_i^*\neq 0$ (hence $\psi(|x_i^*|)> 0$) that
\begin{equation*}
0 = \lambda + \alpha^*_i\cdot\nabla_i f(x^*) \cdot\phi'(\psi(|x^*_i|)).
\end{equation*}
{\color{blue}Multiplying both sides of the above equality by $\psi'(|x_i^*|){\rm sgn}(x_i^*)$ and recalling $\phi = \psi^{-1}$, we obtain}
\begin{equation*}
\begin{split}
0& = \lambda\psi'(|x_i^*|){\rm sgn}(x_i^*) + \alpha^*_i\cdot\nabla_i f(x^*) \cdot\phi'(\psi(|x^*_i|))\psi'(|x_i^*|){\rm sgn}(x_i^*)\\
& = \lambda\psi'(|x_i^*|){\rm sgn}(x_i^*) + \alpha^*_i\cdot\nabla_i f(x^*) \cdot{\rm sgn}(x_i^*)\\
& \overset{\rm (a)}= \lambda\psi'(|x_i^*|){\rm sgn}(x_i^*) + \nabla_i f(x^*) \in \lambda\partial \psi(|\cdot|)(x^*_i) + \nabla_i f(x^*),
\end{split}
\end{equation*}
where (a) holds because $\alpha_i^* = \sgn(x_i^*)$, and the inclusion holds because $\psi$ is differentiable at $|x_i^*| > 0$.
This proves (i).

Now, suppose in addition that $\psi'_+(0)<\infty$ and let $i$ be such that $x_i^* = 0$. Note that \eqref{opt_i2} holds trivially if $\nabla_i f(x^*) = 0$. On the other hand, when $\nabla_i f(x^*)\neq 0$, we have from the first-order optimality conditions for \eqref{optimalityvv} that
\begin{equation}\label{relation3}
\begin{aligned}
  0 &\in \lambda + \alpha_i^*\nabla_i f(x^*)\phi_+'(0) + N_{\R_+}(0)\\
  & = \lambda -|\nabla_i f(x^*)|(\psi'_+(0))^{-1} + N_{\R_+}(0),
\end{aligned}
\end{equation}
where the equality follows from the definition of $\alpha_i^*$ in Definition~\ref{def_ws} (since $\nabla_i f(x^*)\neq 0$) and the facts that $\psi_+'(0)\in (0,\infty)$, $\psi(0) = 0$ and $\phi = \psi^{-1}$ on $\R_+$. Item (ii) now follows immediately from \eqref{relation3} upon recalling that $N_{\R_+}(0) = \R_-$.
\end{proof}

\begin{remark}[Relationship with existing stationarity]
{\color{blue}We discuss the relationship between $\psi_{\rm opt}$-stationary (Definition~\ref{def_ws}) and some existing concepts of stationarity. }
\begin{enumerate}[{\rm (i)}]

\item When $\Omega = \R_+^n$ and $\psi'_+(0) < \infty$, one can see from \cite[{\color{blue}Lemma~2.2 (ii)}]{YuPong19} that
$\partial \psi(|\cdot|)(0) = [-\psi_+'(0),\,\psi_+'(0)]$. {\color{blue}Thus, in this case, if $x^*$ is $\psi_{\rm opt}$-stationary for \eqref{Prob} under Assumption~\ref{Assump1}, then} Proposition~\ref{re_to_st}~(ii) implies that for all $i$ with $x_i^* = 0$,
\[
0 \in  \lambda\partial \psi(|\cdot|)(0) + \nabla_i f(x^*).
\]
In view of this, Proposition~\ref{re_to_st}~(i) and invoking \cite[Exercise~8.8(c)]{RockWets98}, we see that $x^*$ satisfies the following standard first-order optimality condition:
\[
0 \in \nabla f(x^*) + \lambda\partial\Psi(|\cdot|)(x^*) = \partial(f(\cdot) + \lambda\Psi(|\cdot|))(x^*).
\]
{\color{blue}Moreover, in this case,  we have that $f(\cdot) + \Psi(|\cdot|)$ is a difference-of-convex function and $\Psi(|\cdot|)$ is regular. Consequently,
\begin{equation*}
0 \in \nabla f(x^*) + \lambda\partial\Psi(|\cdot|)(x^*) = \nabla f(x^*) + \lambda\widehat{\partial}\Psi(|\cdot|)(x^*) \subseteq \widehat{\partial}\left(f(\cdot) + \lambda\Psi(|\cdot|)\right)(x^*).
\end{equation*}
This means that $x^*$ is a d-stationary point, in view of the definition on \cite[page 28]{LiSoMa20}.
}
\item  {\color{blue}When $\Omega = \R_+^n$ and $\psi_+'(0) = \infty$,  our model \eqref{Prob}  is a special case of the model in \cite{BC17}, in which a generalized stationary point was defined for constrained problems with a non-Lipschitz objective function and a closed convex constraint.
We show that, in this case, if $x^*$ is $\psi_{\rm opt}$-stationary for \eqref{Prob} under Assumption~\ref{Assump1}, then $x^*$ is a generalized stationary point of \eqref{Prob}.

To this end, we first note from \cite[Definition~2]{BC17} that an $\bar{x}$ is a generalized stationary point of \eqref{Prob} under Assumption~\ref{Assump1} with $\Omega = \R_+^n$ and $\psi_+'(0) = \infty$ if $H^{\circ}(\bar{x}; v; \R^n)\ge 0$ for every $v\in V_{\bar{x}}:=\{w:\; w_i = 0 \mbox{ if } \bar{x}_i = 0\}$, where $H(\cdot) = f(\cdot) + \lambda\Psi(|\cdot|)$ and  $H^{\circ}(\bar{x}; v; \R^n):=\underset{y\to\bar{x}, t\downarrow 0}{\lim\sup} \frac{H(y + tv) - H(y)}{t}$.
Now, notice that for every $v\in V_{\bar{x}}$,
\begin{equation*}
\begin{split}
H^{\circ}(\bar{x}; v; \R^n) & = \underset{y\to\bar{x}, t\downarrow 0}{\lim\sup} \frac{f(y + tv) - f(y) + \lambda\sum_{i=1}^n[\psi(|y_i + tv_i|) - \psi(|y_i|)]}{t}\\
& \overset{\rm (a)}{=} \underset{y\to\bar{x}, t\downarrow 0}{\lim\sup} \frac{f(y + tv) - f(y) + \lambda\underset{i: \bar{x}_i\neq 0}{\sum}[\psi(|y_i + tv_i|) - \psi(|y_i|)]}{t}\\
& \overset{\rm (b)}{=} \langle \nabla f(\bar{x}),\,v\rangle + \lambda\sum_{i:\bar{x}_i\neq 0}\underset{y_i\to\bar{x}_i, t\downarrow 0}{\lim\sup}\frac{\psi(|y_i + tv_i|) - \psi(|y_i|)}{t}\\
& \overset{\rm (c)}{=} \langle \nabla f(\bar{x}),\,v\rangle + \lambda\sum_{i:\bar{x}_i\neq 0}\psi'(|\bar{x}_i|){\rm sgn}(\bar{x}_i) v_i\\
& \overset{\rm (d)}{=} \sum_{i:\bar{x}_i\neq 0}\left(\nabla_if(\bar{x}) + \lambda\psi'(|\bar{x}_i|){\rm sgn}(\bar{x}_i)\right)v_i,
\end{split}
\end{equation*}
where {\rm (a)} follows from the definition of $V_{\bar{x}}$, {\rm (b)} follows from the smoothness of $f$, {\rm (c)} follows from the differentiability of $\psi$ on $(0,\,\infty)$ and {\rm (d)} follows from the fact that $v\in V_{\bar{x}}$.
Thus, $H^{\circ}(\bar{x}; v; \R^n) \ge 0$ for every $v\in V_{\bar{x}}$ means that for each $i$ with $\bar{x}_i\neq 0$,
\begin{equation*}
\left(\nabla_if(\bar{x}) + \lambda\psi'(|\bar{x}_i|){\rm sgn}(\bar{x}_i)\right)v_i \ge 0 \ \ \mbox{for all } v_i\in\R.
\end{equation*}
This is further equivalent to
\begin{equation*}
 0 = \nabla_i f(\bar{x}) + \lambda\psi'(|\bar{x}_i|){\rm sgn}(\bar{x}_i) \in \nabla_if(\bar{x}) + \lambda\partial\psi(|\cdot|)(\bar{x}_i).
\end{equation*}
Consequently, $x^*$ being a generalized stationary point is equivalent to \eqref{opt_i}, which is implied by $x^*$ being $\psi_{\rm opt}$-stationary, thanks to Proposition~\ref{re_to_st}~(i).
}

As a specific example, when $f(x) = \|Ax - b\|^2$ and $\psi(t) = t^p$ with $p \in (0,0.5]$, condition \eqref{opt_i} can be written as
\begin{equation*}
0 = \lambda p|x_i^*|^{p-1}\sgn(x_i^*) + 2\left(A^\top(Ax^* - b)\right)_i, \ \ \forall\,i\ {\color{blue} with}\ x_i^*\neq 0.
\end{equation*}
Let $X^*$ be the diagonal matrix whose $i$th diagonal entry equals $x_i^*$. The above display can be further equivalently written as
\begin{equation*}
0 = \lambda p|x^*|^p + 2X^*A^\top(Ax^* - b),
\end{equation*}
which reduces to the standard first-order optimality condition for the optimization problem $\min_{x\in \R^n}f(x) + \lambda\sum_{i=1}^n|x_i|^p$; see, for example, \cite[Definition~3.1]{CXY10}.

\end{enumerate}

\end{remark}

\section{Algorithm and convergence analysis}\label{sec4}

In this section, we motivate and present our algorithm for solving \eqref{Prob} under Assumption~\ref{Assump1}, and establish subsequential convergence of our proposed algorithm to $\psi_{\rm opt}$-stationary points.

Noting that the objective of \eqref{Prob} consists of a smooth part $f$ (with Lipschitz gradient) and a nonsmooth part $(
\lambda\Psi + \delta_{\Omega})(|\cdot|)$, it is tempting to adapt the proximal gradient algorithm, which is a popular class of algorithm for tackling optimization problems with objectives being the sum of a smooth part and a nonsmooth part. However, suppose we directly apply the proximal gradient algorithm with constant stepsize $\gamma\in (0,\frac1{L_f})$, we will be confronted with the following subproblem in every iteration: given $x^k$, the $x^{k+1}$ is obtained as an $\breve x$ satisfying
\begin{equation}\label{PGsubproblem}
\breve x \in \Argmin_{|x|\in \Omega}\left\{\langle\nabla f(x^k),x-x^k\rangle+\frac1{2\gamma}\|x - x^k\|^2 + \lambda\Psi(|x|)\right\}.
\end{equation}
This subproblem basically requires computing the so-called proximal mapping of the nonsmooth nonconvex function $x\mapsto \gamma\cdot(\lambda\Psi + \delta_{\Omega})(|x|)$, which does not have closed-form solutions in general. Thus, it appears that the proximal gradient algorithm cannot be efficiently applied to solving \eqref{Prob}.

Despite not having closed-form solutions, the subproblem \eqref{PGsubproblem} looks highly structured. Indeed, note that \eqref{PGsubproblem} can be equivalently written as
\begin{equation}\label{subprob_orig}
\breve x \in \Argmin_{|x|\in \Omega}\left\{\frac1{2\gamma}\|x - (x^k - \gamma \nabla f(x^k))\|^2 + \lambda\Psi(|x|)\right\}.
\end{equation}
Based on this reformulation and Remark~\ref{rem:converse}, we see that a solution $\breve x$ of \eqref{PGsubproblem} can be obtained as $\breve x = \breve\alpha\circ \breve w$, where $\breve\alpha = {\rm sgn}(x^k - \gamma \nabla f(x^k))$ and
\begin{equation*}
\breve w\in \Argmin_{w\in \Omega} \left\{\frac1{2\gamma}\|w - |x^k - \gamma \nabla f(x^k)|\|^2 + \lambda\Psi(w)\right\}.
\end{equation*}
The above optimization problem does not seem to be easier to solve compared with \eqref{subprob_orig}, because the projection onto $\Omega$ may not be efficiently executable and the structure of $\psi$ can be complex.\footnote{When $P_{\Omega}$ can be efficiently computed and $\psi(t) =  t$, one can compute $\breve w$ efficiently as an element of $P_{\Omega}(|x^k - \gamma \nabla f(x^k)| - \lambda\gamma e)$. In this case, the proximal gradient algorithm \eqref{PGsubproblem} and its variants can be applied efficiently. See also Remark~\ref{remark:I}.} To further simplify the subproblem we need to solve, we exploit {\color{blue}Assumption~\ref{Assump1} (b)}, which states that $\psi$ has an inverse $\phi$ whose directional derivative is locally Lipschitz, to deduce that a $\breve w$ can be obtained as $\breve w = \phi(\breve v)$, with $\breve v$ given by
\begin{equation}\label{Probv}
\breve v \in \Argmin_{v\in \psi(\Omega)}  \left\{\frac1{2\gamma}\|\phi(v) - |x^k - \gamma \nabla f(x^k)|\|^2 + \lambda \langle e,v\rangle\right\}.
\end{equation}
Such a reparametrization strategy was also used recently in \cite{LiMcKenzieYin21,XiaoBai21} in the special case when $\psi(t) = \sqrt{t}$ (i.e., $\phi(t) = t^2$) for some simplex-constrained problems, and was called Hadamard parametrization in \cite{LiMcKenzieYin21}. In principle, the optimization problem in \eqref{Probv} can be solved approximately by the gradient projection algorithm (despite the fact that the objective is only continuously differentiable in $\R^n_{++}$), because projections onto $\psi(\Omega)$ {\color{blue}are} easy to compute by assumption. Then one can obtain $\breve x$ approximately as ${\rm sgn}(x^k - \gamma \nabla f(x^k)) \circ \phi(\breve v)$.

However, solving the subproblem~\eqref{Probv} to a desired accuracy can be time consuming. Having this in mind, our algorithm, which is presented in Algorithm~\ref{Alg1} below, is essentially based on solving the proximal gradient subproblem \eqref{PGsubproblem} ``roughly" that we apply only {\em one step} of gradient projection to \eqref{Probv}. Since the objective of \eqref{Probv} does not have globally Lipschitz gradient, we incorporate a linesearch scheme in Step 1b) to search for a viable parameter $\widetilde \eta$. We also incorporate a standard non-monotone linesearch scheme \eqref{ls} to look for a viable $\widetilde\gamma$. Observe that in this algorithm, we maintain $v^k = \psi(|x^k|)$ (and hence $|x^k| = \phi(v^k)$) for all $k\ge 0$.

\begin{algorithm}[h]
	\caption{Doubly majorized algorithm (DMA) for \eqref{Prob} under Assumption~\ref{Assump1}}\label{Alg1}
	\begin{algorithmic}
		\STATE
		\begin{description}
		  \item[\bf Step 0.] Take any $x^0$ with $|x^0|\in \Omega$. Let $\gamma_{\max}\ge\gamma_{\min} > 0$ and  $0 < \underline{\eta}<\overline{\eta}<\infty$. Let $c_1 > 0$, $\tau\in(0,\,1)$ and pick an integer $M \ge 0$. Let $v^0 = \psi(|x^0|)$ and set $k = 0$.
		  \item[\bf Step 1.] Pick any $\widetilde \gamma\in[\gamma_{\min},\,\gamma_{\max}]$.
		  \begin{description}
		   \item[\bf 1a)] Pick any $\widetilde\eta\in [\underline{\eta},\overline{\eta}]$. Consider $G_{\widetilde\gamma}(v) := \lambda \langle e,v\rangle + \sum_{i=1}^n g^i_{\widetilde\gamma}(v_i)$ with $g^i_{\widetilde\gamma}(v_i):= \frac1{2\widetilde \gamma}(\phi(v_i) - |x_i^k - \widetilde \gamma \nabla_i f(x^k)|)^2$ for each $i$.
		  \item[\bf 1b)] Compute
\[
\widetilde v \in \Argmin_{v\in \psi(\Omega)}\left\{ \frac{\widetilde\eta}{2}\|v - v^k\|^2 + \sum_{i=1}^n[\lambda +(g^i_{\widetilde\gamma})'_+(v_i^k)]\cdot (v_i - v_i^k)\right\}.
\]
If $G_{\widetilde\gamma}(\widetilde v)\le G_{\widetilde\gamma}(v^k)$, go to \textbf{Step 1c)}; otherwise, update $\widetilde\eta \leftarrow \widetilde\eta/\tau$ and go to  \textbf{Step 1b)}.
		  \item[\bf 1c)] Set $\widetilde u = \sgn(x^k - \widetilde \gamma \nabla f(x^k))\circ \phi(\widetilde v)$. If
		  \begin{equation}\label{ls}
		  F(\widetilde u) \le \max_{[k - M]_+\le i\le k}F(x^i) - \frac{c_1}{2}\|\widetilde u - x^k\|^2,
          \end{equation}
go to  \textbf{Step 2}; otherwise, update $\widetilde \gamma\leftarrow\tau\widetilde \gamma$ and go to  \textbf{Step 1a)}.
		 \end{description}
		
		 \item[\bf Step 2.] Set $\widebar{\eta}_k = \widetilde\eta$, $\gamma_k = \widetilde \gamma$, $v^{k+1} = \widetilde v$ and $x^{k+1} =\widetilde u$. Update $k\leftarrow k + 1$ and go to \textbf{Step 1}.
		\end{description}
	\end{algorithmic}
\end{algorithm}
\begin{remark}\label{remark:I} We have the following observations concerning Algorithm~\ref{Alg1} when $\psi(t) = t$.
\begin{enumerate}[{\rm (i)}]
  \item If $\psi(t) = t$ and the $\widetilde{\eta}$ in {\em Step 1a)} is chosen such that $\widetilde{\eta} \ge \frac{1}{\widetilde{\gamma}}$, then \emph{Step 1b)} will be invoked exactly once per iteration. Indeed, when $\psi(t) = t$, we have $\phi(t) = t$ and hence each $g^i_{\widetilde\gamma}$ is continuously differentiable on $\R$. Thus, the subproblem in \emph{Step 1b)} can be rewritten as
\begin{equation}\label{rew_subp}
\widetilde{v} \in\Argmin_{v\in\Omega}\left\{\frac{\widetilde{\eta}}{2}\|v - v^k\|^2 + \langle\nabla G_{\widetilde{\gamma}}(v^k),\, v - v^k\rangle\right\}.
\end{equation}
Note that $v^k\in\Omega$ and $\nabla G_{\widetilde{\gamma}}$ has Lipschitz modulus $\frac{1}{\widetilde{\gamma}}$. Using these, $\widetilde{\eta} \ge \frac{1}{\widetilde{\gamma}}$ and the fact that $\widetilde{v}$ is a minimizer in \eqref{rew_subp}, we have
\begin{eqnarray*}
G_{\widetilde{\gamma}}(\widetilde{v}) \!\!\!\!\!\!\!\!\!
&& \displaystyle \le G_{\widetilde{\gamma}}(v^k) + \langle\nabla G_{\widetilde{\gamma}}(v^k),\, \widetilde{v} - v^k\rangle + \frac{1}{2\widetilde{\gamma}}\|\widetilde{v} - v^k\|^2 \\
&&  \displaystyle \le G_{\widetilde{\gamma}}(v^k) + \langle\nabla G_{\widetilde{\gamma}}(v^k),\, \widetilde{v} - v^k\rangle + \frac{\widetilde{\eta}}{2}\|\widetilde{v} - v^k\|^2 \\
&&  \displaystyle \le G_{\widetilde{\gamma}}(v^k) + \langle\nabla G_{\widetilde{\gamma}}(v^k),\, v^k - v^k\rangle + \frac{\widetilde{\eta}}{2}\|v^k - v^k\|^2  = G_{\widetilde{\gamma}}(v^k).
\end{eqnarray*}
Hence, {\em Step 1b)} is invoked exactly once because $\widetilde\eta$ does not need to be updated.
\item If $\psi(t) = t$ and the $\widetilde{\eta}$ in {\em Step 1a)} is chosen as $\frac{1}{\widetilde{\gamma}}$ in every iteration, then \emph{Algorithm~\ref{Alg1}} reduces to a proximal gradient algorithm with non-monotone linesearch {\color{blue}(NPG)}. To see this, first observe that the subproblem in {\rm Step 1b)} can be further rewritten from \eqref{rew_subp} to
\begin{equation*}
\widetilde{v}\in P_{\Omega}\Big(v^k - \frac{1}{\widetilde{\eta}}\nabla G_{\widetilde{\gamma}}(v^k)\Big) = P_{\Omega}\Big(v^k - \big(v^k - y^k + \widetilde{\gamma}\lambda e\big)\Big) = P_{\Omega}\big(y^k - \widetilde{\gamma}\lambda e\big),
\end{equation*}
where $y^k := |x^k - \widetilde\gamma\nabla f(x^k)|$.
Using this observation, the definition of $\widetilde u$ in {\rm Step 1c)} and Remark~\ref{rem:converse}, we conclude that $\widetilde{u}$ satisfies
\begin{equation*}
\widetilde{u} \in \Argmin_{|x|\in \Omega}\left\{\frac1{2\widetilde\gamma}\|x - (x^k - \widetilde\gamma \nabla f(x^k))\|^2 + \lambda\sum_{i=1}^n|x_i|\right\}.
\end{equation*}
This together with \eqref{ls} shows that \emph{Algorithm~\ref{Alg1}} reduces to {\color{blue}NPG} in this case.
\end{enumerate}
\end{remark}

We next establish the well-definedness of Algorithm~\ref{Alg1}. Specifically, we will argue that Step 1b) and Step 1c) are invoked finitely many times in each iteration. To this end, consider  \eqref{Prob} and suppose that Assumption~\ref{Assump1} holds.
Fix any $\gamma > 0$ and $\underline{\eta} > 0$. For each fixed $\eta \ge \underline{\eta}$ and $\widehat x\in \R^n$ with $\widehat v:= \psi(|\widehat x|)\in \psi(\Omega)$, define
\begin{equation}\label{defG}
    G_\gamma(v) := \lambda \langle e,v\rangle + \sum_{i=1}^n g_\gamma^i(v_i)
\end{equation}
with $g_\gamma^i(t) := \frac1{2\gamma}(\phi(t) - |\widehat x_i - \gamma \nabla_i f(\widehat x)|)^2$ for each $i$, and let $v_\eta$ be any element such that
\begin{equation}\label{defv}
v_\eta \in \Argmin_{v\in \psi(\Omega)}\left\{\frac{\eta}{2}\|v - \widehat v\|^2 + \sum_{i=1}^n[\lambda +(g_\gamma^i)'_+(\widehat v_i)]\cdot (v_i - \widehat v_i)\right\}.
\end{equation}
Then we have the following result concerning $v_\eta$.
\begin{lemma}\label{Lemmadescent}
  Consider \eqref{Prob} and suppose that Assumption~\ref{Assump1} holds. Fix any $\gamma > 0$, $\underline{\eta} > 0$, $\eta \ge \underline{\eta}$ and $\widehat x\in \R^n$ with $\widehat v:= \psi(|\widehat x|)\in \psi(\Omega)$, and define $G_\gamma$ and $v_\eta$ as in \eqref{defG} and \eqref{defv}, respectively. Then the following statements hold.
  \begin{enumerate}[{\rm (i)}]
      \item It holds that $\|v_\eta - \widehat v\| \le 2n\underline{\eta}^{-1}(\lambda + \max_{i}|(g_\gamma^i)'_+(\widehat v_i)|)$.
      \item Let $L^i_{\widehat v}$ denote the corresponding $L$ obtained by applying Lemma~\ref{lem:Taylor} with $g = g_{\gamma}^i$ and $a = \widehat v_i + 2n\underline{\eta}^{-1}(\lambda + \max_{j}|(g_{\gamma}^j)'_+(\widehat v_j)|)$ for each $i$. Then it holds that
      \[
      G_\gamma(v_\eta) \le G_\gamma(\widehat v) - \frac{\eta - \max_i L^i_{\widehat v}}2\|v_\eta - \widehat v\|^2.
      \]
      \item Suppose that $v_\eta$ satisfies $G_\gamma(v_\eta) \le G_\gamma(\widehat v)$ and let $u_\gamma := {\rm sgn}(\widehat x - \gamma \nabla f(\widehat x))\circ \phi(v_\eta)$. Then it holds that
      \[
      F(u_\gamma) \le F(\widehat x) - \frac12\left(\frac1\gamma - L_f\right)\|u_\gamma - \widehat x\|^2.
      \]
  \end{enumerate}
\end{lemma}
\begin{proof}
  We first prove (i). From the definition of $v_\eta$ in \eqref{defv} and the fact that $\widehat v\in \psi(\Omega)$, we have $\frac{\eta}{2}\|v_\eta - \widehat v\|^2 + \sum_{i=1}^n[\lambda +(g_{\gamma}^i)'_+(\widehat v_i)]\cdot ([v_\eta]_i - \widehat v_i) \le 0$. Rearranging terms, we see further that
  \[
  \begin{aligned}
    \frac{\eta}{2}\|v_\eta - \widehat v\|^2 &\le - \sum_{i=1}^n[\lambda +(g_{\gamma}^i)'_+(\widehat v_i)]\cdot ([v_\eta]_i - \widehat v_i)\\
    & \le \sum_{i=1}^n[\lambda +\max_j|(g_{\gamma}^j)'_+(\widehat v_j)|]\cdot\|v_\eta - \widehat v\|_\infty\\
    & \le n[\lambda +\max_j|(g_{\gamma}^j)'_+(\widehat v_j)|]\cdot\|v_\eta - \widehat v\|.
  \end{aligned}
  \]
  The desired conclusion now follows immediately from the above display and the fact that $\eta\ge \underline{\eta}$.

  We next prove (ii). Let $L^i_{\widehat v}$ denote the corresponding $L$ obtained by applying Lemma~\ref{lem:Taylor} with $g = g_{\gamma}^i$ and $a = \widehat v_i + 2n\underline{\eta}^{-1}(\lambda + \max_{j}|(g_{\gamma}^j)'_+(\widehat v_j)|)$, and define $I_i:= [0,\widehat v_i + 2n\underline{\eta}^{-1}(\lambda + \max_{j}|(g_{\gamma}^j)'_+(\widehat v_j)|)]$ for each $i$. Then we have $[v_\eta]_i \in I_i$ from item (i). Hence, we have in view of \eqref{defG} and Lemma~\ref{lem:Taylor} that,
  \[
  \begin{aligned}
  &G_{\gamma}(v_\eta) = \lambda \langle e,v_{\eta}\rangle + \sum_{i = 1}^n g_{\gamma}^i([v_\eta]_i) = \lambda \langle e, \widehat v\rangle + \lambda \langle e, v_\eta - \widehat v\rangle + \sum_{i = 1}^n g_{\gamma}^i([v_\eta]_i) \\
  & \le \lambda \langle e,\widehat v\rangle + \sum_{i = 1}^n\left[g_{\gamma}^i(\widehat v_i) + [\lambda +(g_{\gamma}^i)'_+(\widehat v_i)]([v_\eta]_i - \widehat v_i) + \frac{L^i_{\widehat v}}2 ([v_\eta]_i - \widehat v_i)^2\right]\\
  & \overset{\rm (a)}\le
  \lambda \langle e, \widehat v\rangle + \sum_{i = 1}^n\left[g_{\gamma}^i(\widehat v_i) + \frac{L^i_{\widehat v} - \eta}2 ([v_\eta]_i - \widehat v_i)^2\right] \le G_{\gamma}(\widehat v) + \frac{\max_i L^i_{\widehat v} - \eta}2 \|v_\eta - \widehat v\|^2,
  \end{aligned}
  \]
  where (a) follows from the definition of $v_{\eta}$ in \eqref{defv} and the fact that $\widehat v\in \psi(\Omega)$, and the last inequality follows from \eqref{defG}. This proves (ii).

  Finally, we prove (iii).
 Using Taylor's inequality and the fact that $f$ has Lipschitz gradient with modulus $L_f$, we have
  \begin{eqnarray*}
    &&F(u_{\gamma})\le f(\widehat x) + \langle \nabla f(\widehat x),\,u_{\gamma}-\widehat x\rangle+ \frac{L_f}2\|u_{\gamma}-\widehat x\|^2 + (\lambda \Psi+\delta_\Omega)(|u_{\gamma}|)\\
    &&= f(\widehat x) + \langle \nabla f(\widehat x),\,u_{\gamma}-\widehat x\rangle+ \frac1{2\gamma}\|u_{\gamma}-\widehat x\|^2 + (\lambda \Psi+\delta_\Omega)(|u_{\gamma}|) - \beta\|u_{\gamma}-\widehat x\|^2,
  \end{eqnarray*}
  where $\beta := \frac12\left(\frac{1}{\gamma}-L_f\right)$. Rearranging terms in the above display, we obtain
  \begin{eqnarray}\label{haha}
    &&F(u_{\gamma}) + \beta\|u_{\gamma}-\widehat x\|^2 \nonumber \\
    &&\le  f(\widehat x) + \langle \nabla f(\widehat x),\,u_{\gamma}-\widehat x\rangle+ \frac1{2\gamma}\|u_{\gamma}-\widehat x\|^2 + (\lambda \Psi+\delta_\Omega)(|u_{\gamma}|)\nonumber\\
    && = f(\widehat x) - \frac{\gamma}2\|\nabla f(\widehat x)\|^2 + \frac1{2\gamma}\|u_{\gamma}-\widehat x + \gamma\nabla f(\widehat x)\|^2 +  (\lambda \Psi+\delta_\Omega)(|u_{\gamma}|)\nonumber\\
    && \overset{\rm (a)}= f(\widehat x) - \frac{\gamma}2\|\nabla f(\widehat x)\|^2 + \frac1{2\gamma}\|\widehat\alpha\circ\phi(v_\eta)-\widehat\alpha\circ|\widehat x - \gamma\nabla f(\widehat x)|\|^2 +  (\lambda \Psi+\delta_\Omega)(\phi(v_\eta))\nonumber\\
    && = f(\widehat x) - \frac{\gamma}2\|\nabla f(\widehat x)\|^2 + \frac1{2\gamma}\|\phi(v_\eta)-|\widehat x - \gamma\nabla f(\widehat  x)|\|^2 + \lambda \sum_{i=1}^n[v_\eta]_i\nonumber\\
    && \overset{\rm (b)}= f(\widehat x) - \frac{\gamma}2\|\nabla f(\widehat x)\|^2 + G_{\gamma}(v_\eta),
  \end{eqnarray}
  where (a) holds because $|u_{\gamma}|= \phi(v_\eta)$ componentwise and we write $\widehat\alpha := {\rm sgn}(\widehat x - \gamma \nabla f(\widehat x))$ for notational simplicity, and (b) follows from the definition of $G_{\gamma}$ in \eqref{defG}. Using the assumption that $G_{\gamma}(v_\eta) \le G_{\gamma}(\widehat v)$, we deduce further from \eqref{haha} that
  \begin{equation*}
    \begin{split}
    &F(u_{\gamma}) + \beta\|u_{\gamma}-\widehat x\|^2 \le f(\widehat x) - \frac{\gamma}2\|\nabla f(\widehat x)\|^2 + G_{\gamma}(\widehat v)\\
    & = f(\widehat x) - \frac{\gamma}2\|\nabla f(\widehat x)\|^2 + \frac1{2\gamma}\|\phi(\widehat v)-|\widehat x - \gamma\nabla f(\widehat x)|\|^2 + \lambda \sum_{i=1}^n\widehat v_i\\
    & \le f(\widehat x) - \frac{\gamma}2\|\nabla f(\widehat x)\|^2 + \frac1{2\gamma}\|\widehat x-(\widehat x - \gamma\nabla f(\widehat x))\|^2 + (\lambda \Psi+\delta_\Omega)(|\widehat x|)\\
    & = F(\widehat x),
  \end{split}
  \end{equation*}
 where the last inequality follows from $\widehat v = \psi(|\widehat x|)$ (thus $\phi(\widehat v) = |\widehat x|$), $|\widehat x|\in\Omega$ and the triangle inequality. This completes the proof.
\end{proof}

\begin{remark}[Well-definedness of Algorithm~\ref{Alg1}]\label{out_well}
We discuss the well-definedness of Algorithm~\ref{Alg1}, i.e., we argue that in each iteration, Step 1b) and Step 1c) are only invoked finitely many times.

  Suppose that an $x^k$ is given for some $k\ge 0$. Observe from the update rule of the algorithm that $v^k = \psi(|x^k|)$. For a given $\widetilde\gamma > 0$, by applying {\color{blue}Lemma~\ref{Lemmadescent} (ii)} with $\widehat x = x^k$ and invoking \eqref{defG} and \eqref{defv}, we conclude that $G_{\widetilde \gamma}(\widetilde v) \le G_{\widetilde \gamma}(v^k)$ for all sufficiently large $\widetilde \eta$. This together with the update rule of $\widetilde \eta$ shows that Step 1b) will only be invoked finitely many times given any $\widetilde\gamma$.

  In addition, for any $\widetilde v$ that satisfies $G_{\widetilde\gamma}(\widetilde v) \le G_{\widetilde\gamma}(v^k)$, according to {\color{blue}Lemma~\ref{Lemmadescent} (iii)}, the corresponding $\widetilde u$ will satisfy \eqref{ls} whenever $\widetilde \gamma \le \frac{1}{c_1 + L_f}$. In view of the update rule of $\widetilde \gamma$, we can also conclude that Step 1c) is invoked only finitely many times at the $k$th iteration. This also implies that Step 1b) will only be repeated for finitely many different $\widetilde \gamma$.
These observations together with an induction argument prove the well-definedness of Algorithm~\ref{Alg1}.

Finally, notice that at iteration $k$, the initial $\widetilde\gamma$ at the beginning of Step 1 lies in $[\gamma_{\min},\gamma_{\max}]$. Hence, we conclude based on this and the update rule of $\widetilde\gamma$ that
\[
\gamma_{\max}\ge \gamma_k \ge \min\left\{\gamma_{\min},\frac{\tau}{c_1 + L_f}\right\} =: \widetilde\gamma_{\min}.
\]
\end{remark}

We now show that any accumulation point of the $\{x^k\}$ generated by Algorithm~\ref{Alg1} is a $\psi_{\rm opt}$-stationary point. In this regard, we can say that Algorithm~\ref{Alg1} is a {\em companion} algorithm for the notion of $\psi_{\rm opt}$ stationarity. This companion relationship is not too unexpected upon noting the similarity between the derivations that led to Algorithm~\ref{Alg1} and the proof of Theorem~\ref{thm:opt-to-psiopt} (which establishes the necessity of $\psi_{\rm opt}$-stationarity for global optimality).

\begin{theorem}[Subsequential convergence]\label{thm2}
  Consider \eqref{Prob} and suppose that Assumption~\ref{Assump1} holds. Let $\{x^k\}$ and $\{\widebar{\eta}_k\}$ be generated by Algorithm~\ref{Alg1}. Then the following statements hold.
  \begin{enumerate}[{\rm (i)}]
      \item It holds that $\lim_{k\to\infty}\|x^{k+1} - x^k\| = 0$.
      \item The sequences $\{x^k\}$ and $\{\widebar{\eta}_k\}$ are bounded.
      \item Any accumulation point $x^*$ of $\{x^k\}$ satisfies the $\psi_{\rm opt}$-stationarity condition.
  \end{enumerate}
\end{theorem}
\begin{proof}
First, we see from the criterion \eqref{ls} that for all $k$,
\begin{equation*}
F(x^{k}) \le F(x^0) < \infty.
\end{equation*}
Notice that $F$ is level-bounded because $f$ and $\psi$ are nonnegative functions, $\lambda > 0$, and $\psi$ is level-bounded according to Assumption~\ref{Assump1}. Consequently, the sequence $\{x^k\}$ is bounded. Moreover, the conclusion $\lim_{k\to\infty}\|x^{k+1} - x^k\| = 0$ can be proved similarly as in \cite[Lemma~4]{WNF09}.

We next prove the boundedness of $\{\widebar{\eta}_k\}$. We start by deriving an auxiliary bound on a particular choice of $L^i_{\hat v}$ that satisfies the assumption in {\color{blue}Lemma~\ref{Lemmadescent} (ii)}. To this end, let $M := \sup_{k}\|x^k\|$: this quantity is finite because $\{x^k\}$ is bounded. Fix any $i\in \{1,\ldots,n\}$ and any $k\ge 0$. Then
\begin{equation}\label{bounda}
  \begin{aligned}
  a^k_i&:= v^k_i + 2n\underline{\eta}^{-1}(\lambda + \max_{j}|(g_{\gamma_k}^j)'_+(v^k_j)|) \\
  & \overset{\rm (a)}\le \|v^k\| + 2n \underline{\eta}^{-1} (\lambda + \gamma_k^{-1}\max_j \left|\phi(v_j^k) - |x^k_j - \gamma_k \nabla_j f(x^k)|\right|\cdot|\phi'_+(v_j^k)|)\\
  & \overset{\rm (b)}= \|\psi(|x^k|)\| + 2n \underline{\eta}^{-1} (\lambda + \gamma_k^{-1}\max_j \left||x^k_j| - |x^k_j - \gamma_k \nabla_j f(x^k)|\right|\cdot|\phi'_+(\psi(|x^k_j|))|)\\
  & \le \|\psi(|x^k|)\| + 2n \underline{\eta}^{-1} (\lambda + \max_j |\nabla_j f(x^k)|\cdot|\phi'_+(\psi(|x^k_j|))|)\\
   & \overset{\rm (c)}\le \sup_{\|x\|\le M}\left\{\|\psi(|x|)\| + 2n \underline{\eta}^{-1} (\lambda + \max_j |\nabla_j f(x)|\cdot|\phi'_+(\psi(|x_j|))|)\right\} =: M_1 < \infty,
  \end{aligned}
\end{equation}
  where (a) follows from the definition of $g^i_{\widetilde\gamma}$ in Step 1a) of the algorithm, (b) follows from the definition of $v^k$, (c) holds because $\|x^k\|\le M$, and the finiteness of $M_1$ follows from the continuity of $\psi$, $\phi'_+$ and $\nabla f$. Next, {\color{blue}writing} $y^k:= |x^k - \gamma_k\nabla f(x^k)|$ for notational simplicity, and we apply Lemma~\ref{lem:Taylor0} with $h = \phi$ on $[0,M_1]$ to obtain a $c$ and use it to construct an $L$ as in \eqref{def_L} with $g = g_{\gamma_k}^i$ on $[0,v^k_i + 2n\underline{\eta}^{-1}(\lambda + \max_{j}|(g_{\gamma_k}^j)'_+(v^k_j)|)]$. Denote this $L$ by $L^i_{v^k}$, and observe that this $L^i_{v^k}$ satisfies the assumption in {\color{blue}Lemma~\ref{Lemmadescent} (ii)} with $\gamma = \gamma_k$ and $\widehat v_i = v^k_i$. Moreover, it holds that
  \[
  \begin{aligned}
   & L^i_{v^k} =\frac{c}{\gamma_k}\left(\sup_{t\in [0,a_i^k]}|\phi(t)| + |y^k_i|\right) + \frac1{\gamma_k}\left(\sup_{t\in [0,a_i^k]}|\phi'_+(t)| + \frac{a_i^kc}2\right)^2\\
   &\le \frac{c}{\widetilde\gamma_{\min}}
   \left(\sup_{t\in [0,M_1]}|\phi(t)| + M + \gamma_{\max}\sup_{\|x\|\le M}\|\nabla f(x)\|\right) + \frac1{\widetilde\gamma_{\min}}\left(\sup_{t\in [0,M_1]}|\phi'_+(t)| + \frac{M_1c}2\right)^2 =: M_2 < \infty,
  \end{aligned}
  \]
  where $c$ depends only on $M_1$ and the convex function $\phi$ (since it is obtained by applying Lemma~\ref{lem:Taylor0} with $h = \phi$ on $[0,M_1]$) and is independent of $k$, and the inequality follows from \eqref{bounda} and the facts that $\|x^k\|\le M$ and  $\gamma_{\max}\ge \gamma_k \ge \widetilde\gamma_{\min}$ for all $k$ (thanks to Remark~\ref{out_well}). Since $M_2$ is a constant independent of $i$ and $k$, we conclude further that
  \begin{equation}\label{boundL}
  \max_{1\le i\le n}\sup_k L^i_{v^k}\le M_2.
  \end{equation}

Equipped with \eqref{boundL}, we are now ready to argue the boundedness of $\{\widebar{\eta}_k\}$.
Notice that for each $k$, either $G_{\widetilde\gamma}(\widetilde v)\le G_{\widetilde\gamma}(v^k)$ holds for the first $\widetilde\eta$ used in Step 1b) so that $\widebar{\eta}_k \le \overline\eta$, or Step 1b) is invoked multiple times so that the $\widetilde v$ corresponding to $\tau\widebar{\eta}_k$ still gives $G_{\widetilde\gamma}(\widetilde v)> G_{\widetilde\gamma}(v^k)$. In the latter case, applying {\color{blue}Lemma~\ref{Lemmadescent} (ii)} with $\widehat x = x^k$, $\gamma = \gamma_k$, $\eta = \tau \widebar{\eta}_k$ and recalling that the $L^i_{v^k}$ constructed above satisfies the assumption in {\color{blue}Lemma~\ref{Lemmadescent} (ii)}, we see that this case is possible only if $\tau\widebar{\eta}_k\le \max_i L^i_{v^k}$. Combining the two cases with \eqref{boundL}, we conclude that
  \begin{equation*}
  \widebar{\eta}_k\le \max\{\overline{\eta}, \max_i L^i_{v^k}/\tau\}\le \max\{\overline{\eta},M_2/\tau\}.
  \end{equation*}
This completes the proof of item (ii).

  Finally, we prove item (iii).  Let $x^*$ be an accumulation point of $\{x^k\}$ and define $\alpha^k:= {\rm sgn}(x^k - \gamma_k \nabla f(x^k))$ for notational simplicity. Since $\gamma_{\max}\ge \gamma_k \ge \widetilde\gamma_{\min} > 0$ for all $k$ (see Remark~\ref{out_well}) and $\{\widebar\eta_k\}$ is bounded by item (ii),  by passing to further subsequences if necessary, we may assume without loss of generality that there exist subsequences $\{x^{k_j}\}$, $\{\widebar\eta_{k_j}\}$ and $\{\gamma_{k_j}\}$ such that
  \begin{equation}\label{lim}
    \lim_{j\to\infty}x^{k_j} = x^*,\ \ \lim_{j\to\infty}\widebar\eta_{k_j} = \eta_*,\ \ \lim_{j\to\infty}\alpha^{k_j} = \lim_{j\to \infty}{\rm sgn}(x^{k_j} - \gamma_{k_j} \nabla f(x^{k_j})) = \alpha^*,\ \ \lim_{j\to \infty}\gamma_{k_j} = \gamma_*
  \end{equation}
  for some $\alpha^*\in \{-1,1\}^n$, $\eta_* \ge \underline{\eta} > 0$ and $\gamma_* \ge \widetilde \gamma_{\min} > 0$.
 We then have from direct computation that
 \begin{align}
  &\lim_{j\to\infty} (g^i_{\gamma_{k_j}})'_+(v_i^{k_j}) =  \lim_{j\to\infty} (\phi(v_i^{k_j}) - |x_i^{k_j} - \gamma_{k_j} \nabla_i f(x^{k_j})|)\cdot \phi'_+(v_i^{k_j})/\gamma_{k_j}\nonumber\\
  &\overset{\rm (a)}=  \lim_{j\to\infty} (|x_i^{k_j}|- |x_i^{k_j} - \gamma_{k_j} \nabla_i f(x^{k_j})|)\cdot \phi'_+(v_i^{k_j})/\gamma_{k_j}\nonumber\\
   &=  \lim_{j\to\infty} (|x_i^{k_j+1}|- |x_i^{k_j} - \gamma_{k_j} \nabla_i f(x^{k_j})| + |x_i^{k_j}| - |x_i^{k_j+1}|)\cdot \phi'_+(v_i^{k_j})/\gamma_{k_j} \nonumber\\
   &\overset{\rm (b)}=  \lim_{j\to\infty} (\alpha^{k_j}_i(x_i^{k_j+1}-x_i^{k_j} + \gamma_{k_j} \nabla_i f(x^{k_j})) + |x_i^{k_j}| - |x_i^{k_j+1}|)\cdot \phi'_+(v_i^{k_j})/\gamma_{k_j} \nonumber\\
  & = \alpha^*_i\nabla_i f(x^*)  \phi_+'(\psi(|x^*_i|)), \label{gbar_lim}
  \end{align}
 where (a) follows from $v^k = \psi(|x^k|)$ and $\phi = \psi^{-1}$, (b) is true in view of the definition of $\alpha^k$ and the update rule of $x^{k+1}$, and the last equality follows from item (i), \eqref{lim}, the continuity of $\phi_+'$ and the fact that $v^{k_j} = \psi(|x^{k_j}|)$.

 Now, recall that $x^{k+1}=\alpha^k\circ\phi(v^{k+1})$, and
 \begin{equation}\label{bar_min}
 v^{k+1} \in \Argmin_{v\in \psi(\Omega)}\left\{ \frac{\widebar{\eta}_k}{2}\|v - v^k\|^2 + \sum_{i=1}^n[\lambda +(g^i_{\gamma_k})'_+(v_i^k)]\cdot (v_i - v_i^k)\right\}.
 \end{equation}
 From \eqref{bar_min}, we obtain that for each $j\ge 0$,
  \begin{eqnarray}\label{relationship}
  &&\frac{\widebar{\eta}_{k_j}}2\| v^{k_j+1} - v^{k_j}\|^2 + \sum_{i=1}^n[\lambda + (g^i_{\gamma_{k_j}})'_+(v_i^{k_j})]\cdot(v_i^{k_j+1} - v_i^{k_j})\nonumber\\
  && \le
    \frac{\widebar{\eta}_{k_j}}2\|v - v^{k_j}\|^2 + \sum_{i=1}^n[\lambda + (g^i_{\gamma_{k_j}})'_+(v_i^{k_j})]\cdot(v_i - v_i^{k_j})
  \end{eqnarray}
  whenever $v\in \psi(\Omega)$. Also, notice from \eqref{lim} and item (i) that
  \[
    \lim_{j\to\infty}v^{k_j} = \lim_{j\to\infty}\psi(|x^{k_j}|) = \psi(|x^*|)\ \ {\rm and}\ \
    \lim_{j\to\infty}v^{k_j+1}  =
    \lim_{j\to\infty}\psi(|x^{k_j+1}|) = \psi(|x^*|).
  \]
  Using the above display, \eqref{lim} and \eqref{gbar_lim}, we conclude upon passing to the limit as $j$ goes to infinity in \eqref{relationship} that
  \begin{eqnarray*}
  &&\frac{\eta_*}2\|\psi(|x^*|) - \psi(|x^*|)\|^2 + \sum_{i=1}^n\big[\lambda  + \alpha^*_i\cdot\nabla_i f(x^*)\cdot \phi_+'(\psi(|x^*_i|))\big]\cdot(\psi(|x^*_i|) - \psi(|x^*_i|))\nonumber\\
  && \le
    \frac{\eta_*}2\|\psi(|x|) - \psi(|x^*|)\|^2 + \sum_{i=1}^n\big[\lambda  + \alpha^*_i\cdot\nabla_i f(x^*)\cdot  \phi_+'(\psi(|x^*_i|))\big]\cdot(\psi(|x_i|) - \psi(|x^*_i|))
  \end{eqnarray*}
  whenever $x$ satisfies $|x|\in \Omega$.
  Since
  \[
  |x^*| = \lim_{j\to\infty}|x^{k_j}|\in\Omega,
  \]
  we obtain that
  \begin{equation}\label{x*relation}
      x^* \in \Argmin_{|x|\in \Omega}\left\{\frac{\eta_*}2\|\psi(|x|) - \psi(|x^*|)\|^2 + \sum_{i=1}^n\big[\lambda  + \alpha^*_i\cdot\nabla_i f(x^*) \cdot\phi_+'(\psi(|x^*_i|))\big]\cdot \psi(|x_i|)\right\}.
  \end{equation}
  Finally, since $x^{k+1} = \alpha^k\circ \phi(v^{k+1})$, we have $\alpha_i^k = {\rm sgn}(x_i^{k+1})$ whenever $\psi(|x^{k+1}_i|) = v^{k+1}_i \neq 0$. Then using item (i) and \eqref{lim}, we must also have
  \[
  \alpha^*_i = {\rm sgn}(x^*_i)\ \ {\rm if}\ x^*_i \neq 0.
  \]
  Moreover, we see from \eqref{lim} and the lower boundedness of $\{\gamma_k\}$ in Remark~\ref{out_well} that for all $i$ with $x_i^* = 0$ but $\nabla_i f(x^*)\neq 0$, we have $\alpha_i^* = -\sgn(\nabla_if(x^*))$.
  These conditions on $\alpha^*$ together with \eqref{x*relation} show that $x^*$ is a $\psi_{\rm opt}$-stationary point as desired.
\end{proof}

\section{Numerical experiments}\label{sec5}

In this section, we will conduct numerical experiments for Algorithm~\ref{Alg1} on order-constrained compressed sensing problems and block order-constrained sparse time-lagged regression problems.
All experiments are performed in Matlab R2017b on a 64-bit PC with 2.9 GHz Intel Core i9 6-Core and 32GB of DDR4 RAM.

\subsection{Compressed sensing problems with order constraints}
\label{sec5-1}

We first consider the following order-constrained compressed sensing problems with nonconvex regularizers for recovering sparse signals with an order structure:
\begin{equation}\label{order_cons}
    \begin{split}
        \min_{x\in\R^n} & \ \ \frac{1}{2}\|Ax - b\|^2 + \lambda\sum_{i=1}^n\psi(|x_i|)\\
        {\rm s.t.} & \ \ |x_1|\ge |x_2| \ge \cdots \ge |x_n|,
    \end{split}
\end{equation}
where $A\in\R^{m\times n}$, $b\in\R^m$, $\lambda > 0$ and $\psi(t) = t^p$ with $p\in(0,\, 0.5]$ or $\psi(t) = \log(1 + t/\epsilon)$ with $\epsilon > 0$.

We will solve \eqref{order_cons}  with $\psi(t) = t^p$ ($p\in(0,\,0.5]$) by \textbf{DMA} (Algorithm~\ref{Alg1}), and call this algorithm \textbf{DMA}$_{{\rm lp}}$. To the best of our knowledge, our \textbf{DMA} is the only available algorithm for such a model, due to the presence of both the $\ell_p$ regularizer and the order constraints. As a comparison, we consider three other simpler models:
\begin{itemize}
  \item $\min_{x\in \R^n} \frac12\|Ax - b\|^2 + \lambda \sum_{i=1}^n|x_i|^p$ (\emph{i.e.}, change the order-constrained model \eqref{order_cons} to an unconstrained model);
  \item $\min_{|x_1|\ge |x_2| \ge \cdots \ge |x_n|} \frac12\|Ax - b\|^2 + \lambda \sum_{i=1}^n|x_i|$ (\emph{i.e.}, set $\psi(t) = t$ in \eqref{order_cons});
  \item $\min_{x\in \R^n} \frac12\|Ax - b\|^2 + \lambda \sum_{i=1}^n|x_i|$ (\emph{i.e.}, LASSO).
\end{itemize}
Note that all these three models\footnote{Especially, the second model can be solved by \textbf{NPG} as discussed in Remark~\ref{remark:I}.} can be solved by the \textbf{NPG} proposed in \cite{WNF09} (see also \cite{CLP16,LP17,gist13}). We call the corresponding algorithms \textbf{NPG}$_{{\rm lp}}$, \textbf{NPG}$_{{\rm L1c}}$ and \textbf{NPG}$_{{\rm L1}}$, respectively, and we refer to the above four models as ``$\ell_p$-regularized models".

We  also solve \eqref{order_cons} with  $\psi(t) = \log(1 + t/\epsilon)$  by our \textbf{DMA}, and call this algorithm \textbf{DMA}$_{{\rm log}}$. Similarly, as a comparison, we solve a simpler model $\min_{x\in \R^n} \frac12\|Ax - b\|^2 + \lambda \sum_{i=1}^n\log(1 + |x_i|/\epsilon)$ by \textbf{NPG} and call this algorithm \textbf{NPG}$_{\rm log}$. In the following, we refer to these two models as ``logarithmically regularized models".

\paragraph{Data generation.}  First, we randomly generate an $n$-dimensional  vector with $s$ nonzero entries, which follow i.i.d. standard Gaussian distribution. We let the original signal $x_{\rm true}\in\R^n$ be a reordering of this vector such that its entries are nonincreasing in magnitude. Then, we generate $A\in\R^{m\times n}$ by normalizing each column of a randomly generated matrix that has i.i.d. standard Gaussian elements. Next, we set the measurement vector $b = Ax_{\rm true} + \sigma\varepsilon$, where the noise factor $\sigma > 0$ and the noise vector $\varepsilon\in\R^m$ has i.i.d. standard Gaussian entries.

\paragraph{Algorithm settings.} For \textbf{DMA}, we generate an $n$-dimensional random vector with i.i.d. Gaussian entries and set the initial point $x^0$ as the corresponding reordered vector whose entries are nonincreasing in magnitude. We let $c_1 = 10^{-4}$, $\tau = 0.5$ and $M = 4$. In Step 1, we initialize $\widetilde{\eta} = 1$, and initialize $\widetilde{\gamma} = 1$ for $k = 0$ and
\begin{equation*}
\widetilde{\gamma}: = \min\bigg\{\max\bigg\{\frac{\|x^k - x^{k-1}\|^2}{\|A(x^k - x^{k-1})\|^2},\,10^{-8}\bigg\},\,10^8\bigg\}
\end{equation*}
for $k\ge 1$. In Step 1b), we solve the subproblem by a solver developed from \cite{K64}.\footnote{The matlab code can be found in  \url{https://www.mathworks.com/matlabcentral/mlc-downloads/downloads/submissions/47196/versions/1/previews/improve_JP/toolbox_imp_JP/lsqisotonic.m/index.html}.}

For \textbf{NPG}, we use the same settings as those described in \cite[Section~5]{LP17} and set $P(z) = \lambda\sum_{i=1}^n|z_i|^p$ for \textbf{NPG$_{\rm lp}$}, $P(z) = \lambda\|z\|_1 + \delta_S(z)$ where $S:=\left\{x\in\R^n: |x_1|\ge |x_2| \ge \cdots |x_n|\right\}$ for \textbf{NPG}$_{{\rm L1c}}$, $P(z) = \lambda\|z\|_1$ for \textbf{NPG}$_{{\rm L1}}$, and $P(z) = \lambda\sum_{i=1}^n\log(1 + |z_i|/\epsilon)$ for \textbf{NPG}$_{\rm log}$.

We use the same initial point for all six algorithms and terminate them whenever the running time exceeds some fixed time \emph{maxtime} (seconds).

\paragraph{Test settings.} In our experiments, we set $p = 0.5$, $\epsilon =0.5$ and $\sigma = 0.1$, and consider three triples $(n, m, s) = (2560, 540, 180)$,  $(n, m, s) = (10240, 2160, 720)$ and  $(n, m, s) = (25600, 5400, 1800)$.
 For each triple, we generate 10 random instances as described above. For each instance for the triple $(n, m, s) = (2560, 540, 180)$, we solve the $\ell_p$-regularized models with $\lambda = 5\times 10^{-2}$ and the logarithmically regularized models with $\lambda = 8\times 10^{-2}$, and terminate all algorithms with \emph{maxtime} = 4. For each instance for the triple $(n, m, s) = (10240, 2160, 720)$, we solve the $\ell_p$-regularized models with $\lambda = 8\times 10^{-2}$ and the logarithmically regularized models with $\lambda = 10^{-1}$, and terminate all algorithms with \emph{maxtime} = 16. Finally, for each instance for the triple $(n, m, s) = (25600, 5400, 1800)$, we solve the $\ell_p$-regularized models with $\lambda =  10^{-1}$ and the logarithmically regularized models with $\lambda = 2\times 10^{-1}$, and terminate all algorithms with \emph{maxtime} = 40.

To evaluate the performance of all the algorithms, similar to \cite[Section~5.1]{YL17}, we take a normalized measurement of recovery error with respect to time.
Specifically, for each random instance and each algorithm, we let $e_r(k):= \|x^k - x_{\rm true}\| $ be the recovery error at $x^k$ and define
\begin{equation*}
E(t): = \min\big\{e(k): k\in\{i: T(i) \le t\}\big\}\ \ {\rm with}\ \ e(k):= \frac{e_r(k) - e_r^{\min}}{e_r(0) - e_r^{\min}},
\end{equation*}
where $T(k)$ denotes the total computational time until $x^k$ is obtained, and $e_r^{\min}$ is the minimum recovery error among \emph{all} algorithms at termination for this random instance.

In Figure~\ref{result_randm}, for each triple, we compare the average of $E(t)$ over 10 random instances for all six algorithms. In addition, for the triple $(n, m, s) = (25600, 5400, 1800)$, we plot the first $1980$ entries of the recovered signals obtained from each algorithm for one random instance. As one can see,  \textbf{DMA}$_{\rm lp}$ generally outperforms \textbf{NPG}$_{{\rm lp}}$ and \textbf{NPG}$_{{\rm L1c}}$ in terms of recovery error, which suggests the necessity of using the order constraints and the $\ell_p$ regularizer (instead of the $\ell_1$ regularizer), respectively. Also, the outperformance of  \textbf{DMA}$_{\rm log}$ over \textbf{NPG}$_{\rm log}$ highlights the advantage of incorporating the order constraints into the model as well.
Moreover, compared with \textbf{NPG}$_{{\rm L1}}$, the superiority of \textbf{DMA}$_{\rm lp}$ and \textbf{DMA}$_{\rm log}$ implies that solving order-constrained models with nonconvex regularizers can help improve the recovery error in the case when fewer number of observations are available.

\begin{figure}[h!]
\begin{subfigure}{.5\textwidth}
\centering
\includegraphics[width=.85\linewidth]{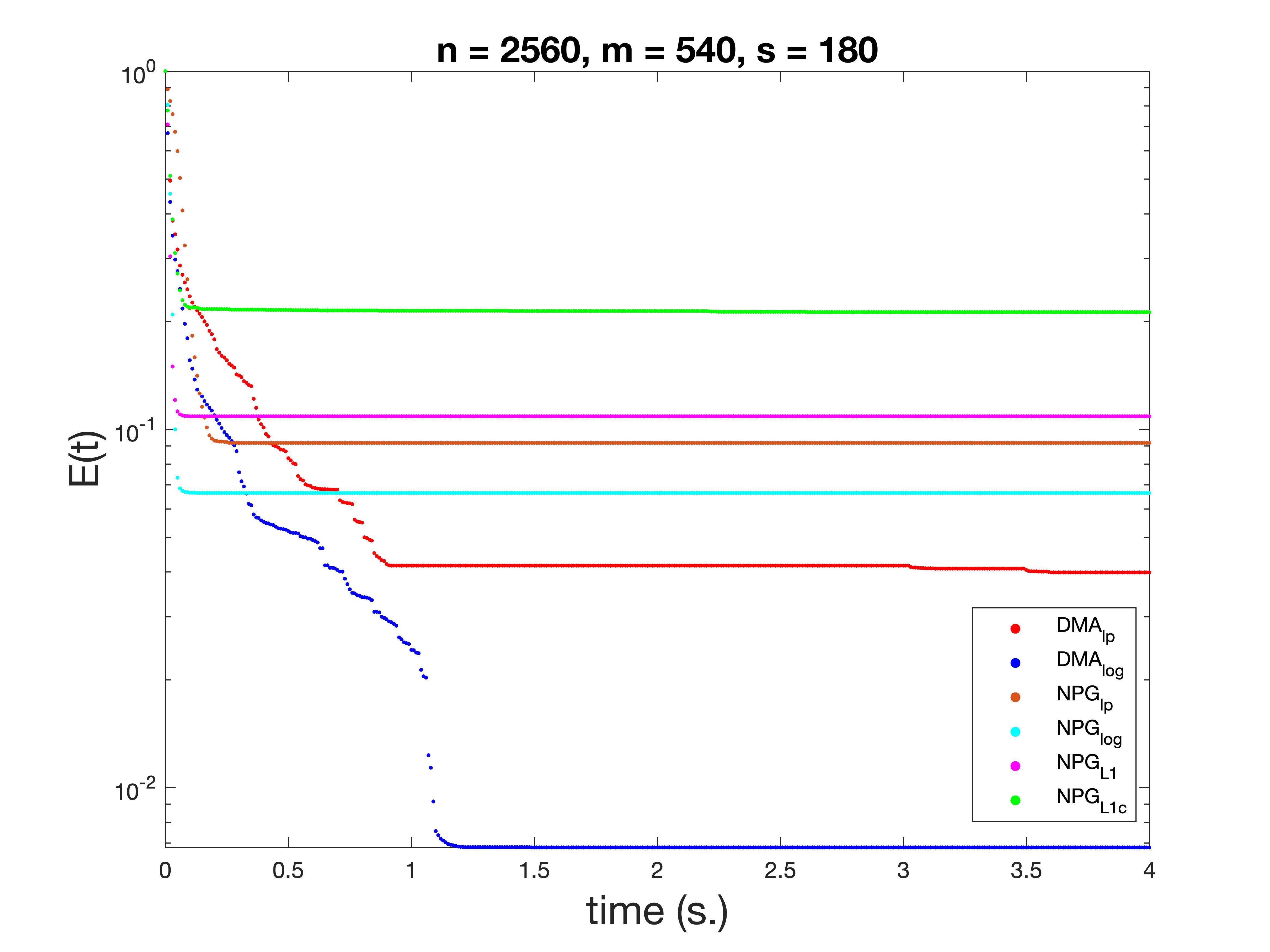}
\end{subfigure}
\hfill
\begin{subfigure}{.5\textwidth}
\centering
\includegraphics[width=.85\linewidth]{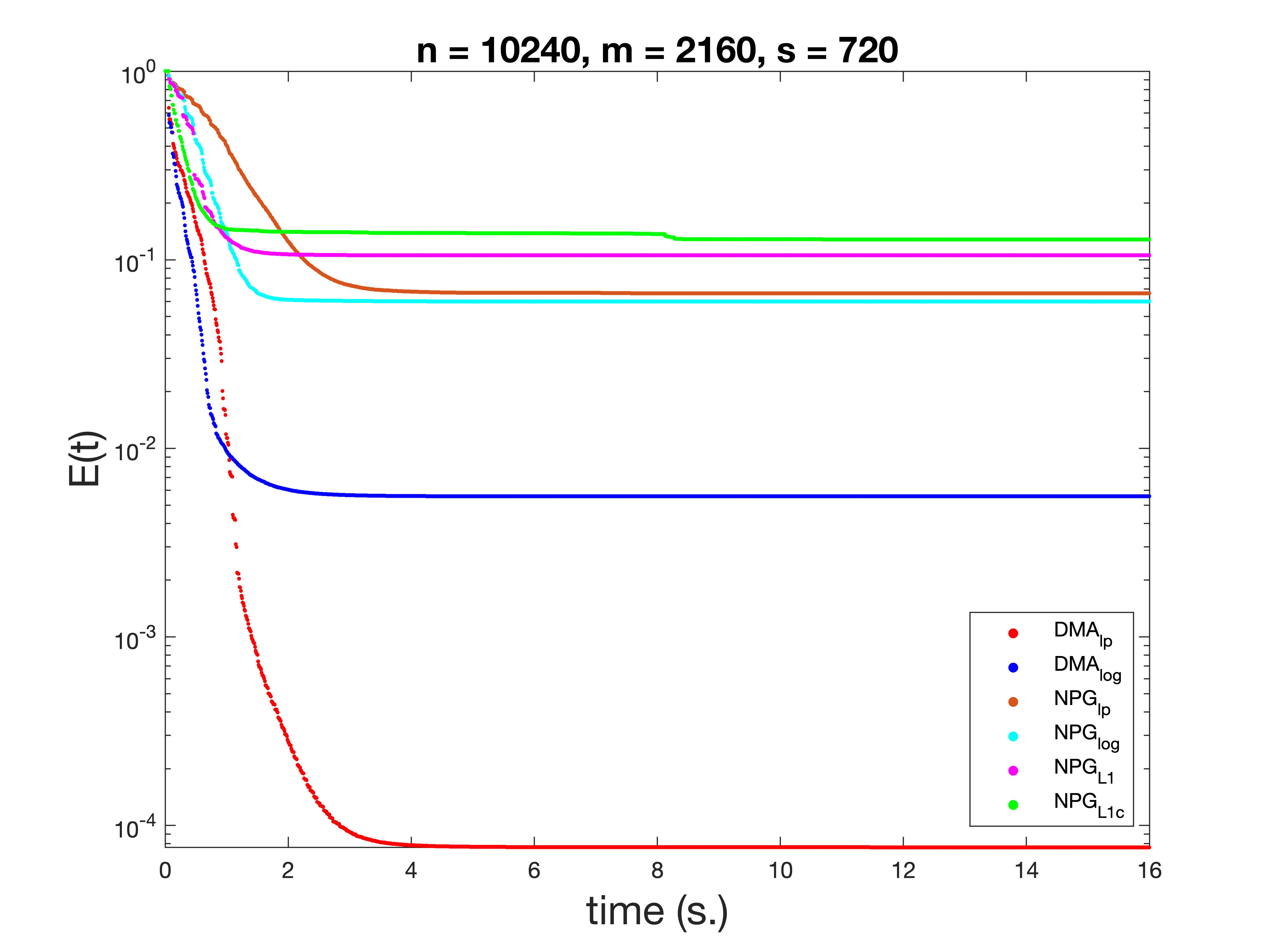}
\end{subfigure}
\hfill
\begin{subfigure}{.5\textwidth}
\centering
\includegraphics[width=.85\linewidth]{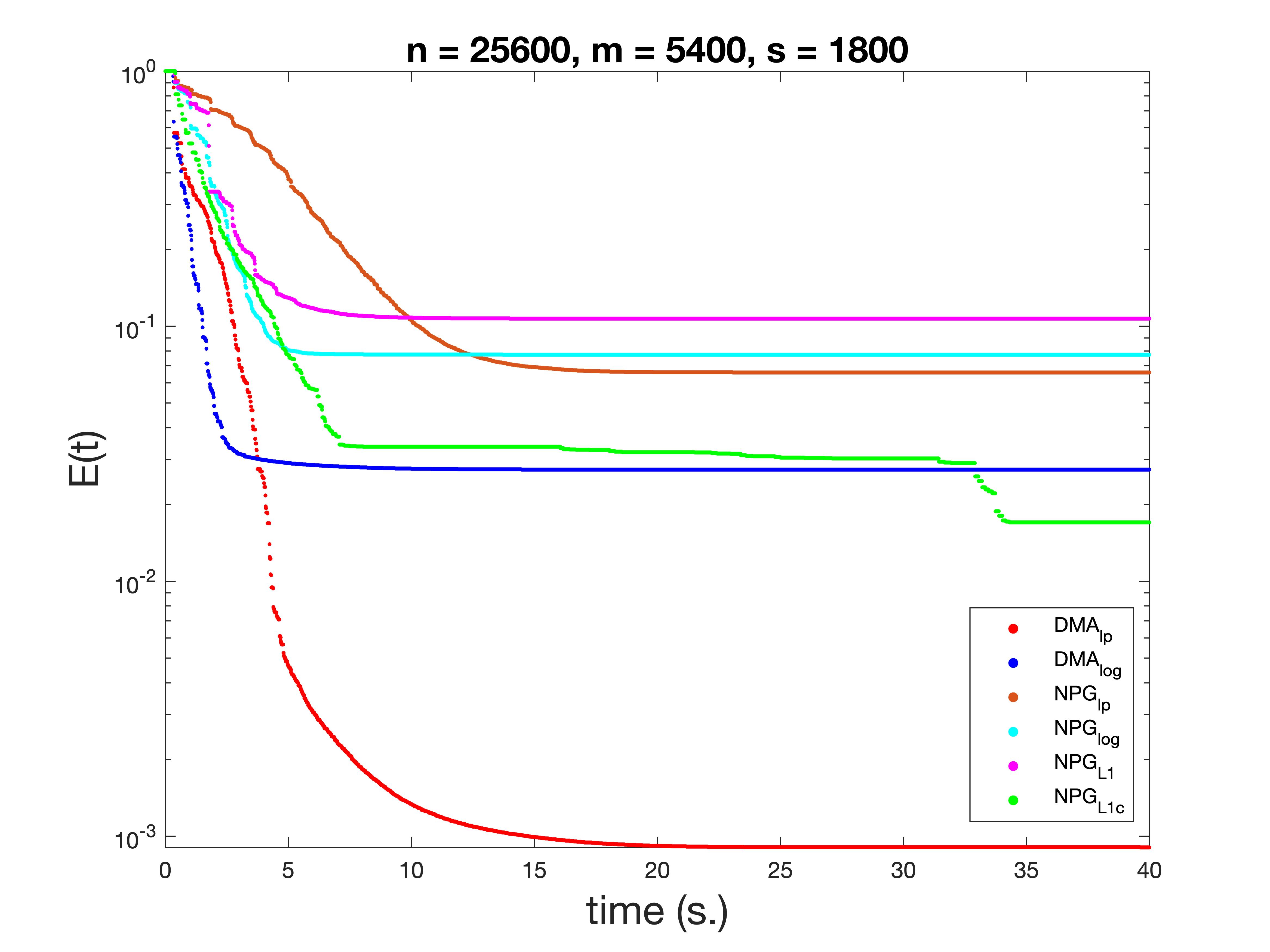}
\end{subfigure}
\hfill
\begin{subfigure}{.5\textwidth}
\centering
\includegraphics[width=.85\linewidth]{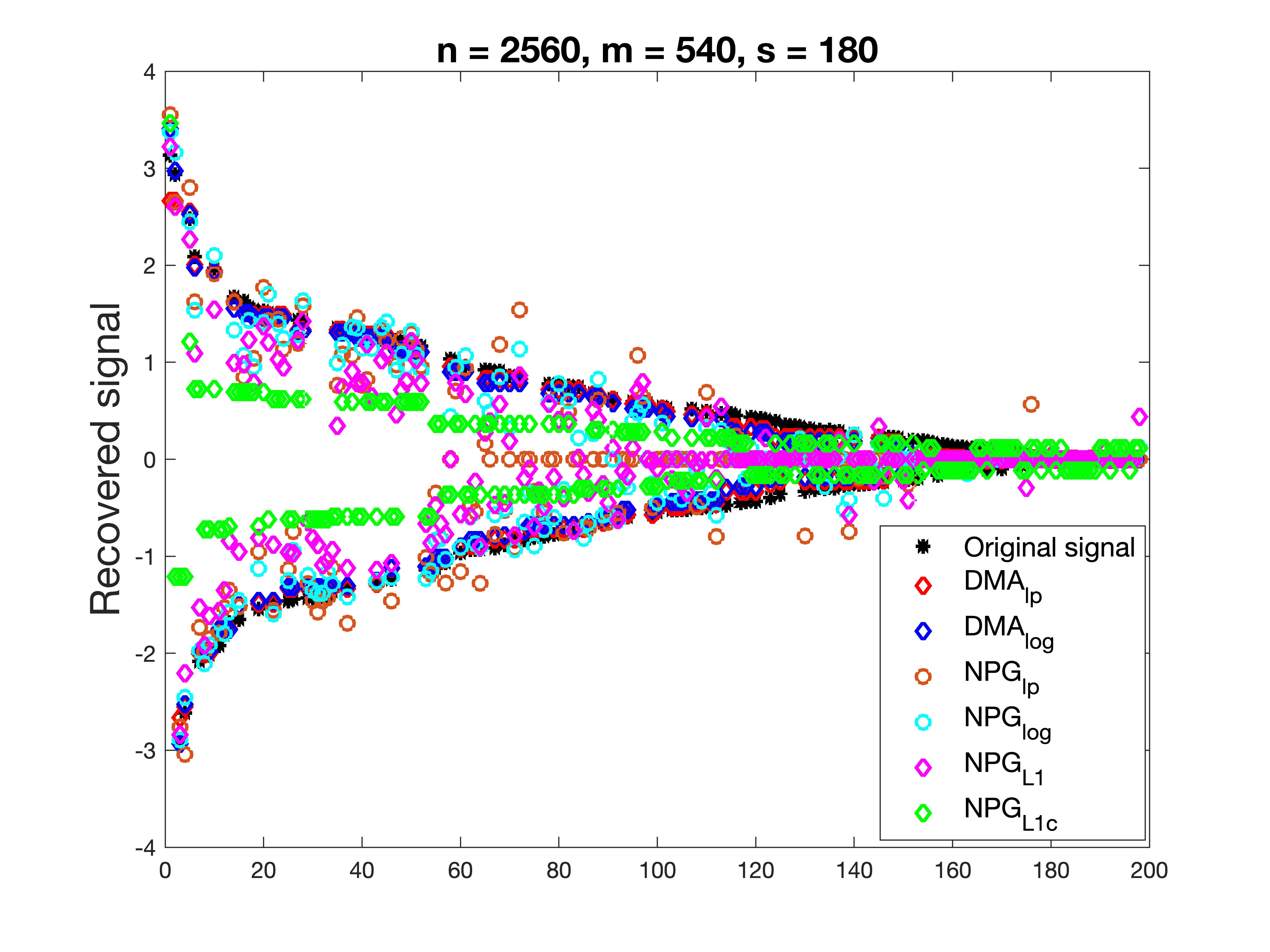}
\end{subfigure}
\caption{Comparison of the averaged recovery errors and the recovered signals (the horizontal axis shows the index $i$ and the vertical axis shows the $i$th entry of the recovered signal). }
\label{result_randm}
\end{figure}

\subsection{Sparse time-lagged regression problems with block order constraints}

We also test our Algorithm~\ref{Alg1} on real data. Specifically, we solve the following block order-constrained model arising from \cite[Section~3]{TibSuo16}\footnote{In \cite{TibSuo16}, the authors only considered the model with $q = 1$ and subsequently solved a convex approximation of it.} for time-lagged regression problems.
\begin{equation}\label{order_blk}
    \begin{split}
        \min_{x\in\R^{pK}} & \ \ \frac{1}{2N}\|Ax - b\|^2 + \lambda\sum_{j=1}^{pK}|x_j|^q\\
        {\rm s.t.} & \ \ |x_{(i -1) K + 1}|\ge |x_{(i -1)K +2}| \ge \cdots \ge |x_{iK}|, \ \ \ i = 1,\ldots,p,
    \end{split}
\end{equation}
where $A\in\R^{N\times pK}$, $b\in\R^N$, $\lambda > 0$ and $q\in(0,\,0.5]\cup\{1\}$. Here, $N$ is the number of observations, $p$ is the number of predictors, and $K$ is the maximum time lag. For $j=1,\ldots,N$, data $b_j$ represents the $j$th observation and data $A_{j,(i-1)K + k}$ represents the value of predictor $i$ of observation $j$ at time-lag $k$ from the current time.

The data we used for test record 330 days of  the level of atmospheric ozone concentration (response variable) and 8 daily meteorological measurements (predictors) made in the Los Angeles basin in 1976; see \url{https://hastie.su.domains/ElemStatLearn/datasets/LAozone.data}.
This data set was used in \cite[Section~3.5]{TibSuo16} and we set a maximum time-lag of 20 days as in \cite[Section~3.5]{TibSuo16}, and predict from the measurements on the current day and the previous 19 days. Then we set both the training and validation sets to have the same size and use cross validation to search for a viable $\lambda$ for final comparison. Specifically, in model \eqref{order_blk}, we let $K = 20$, $p = 8$ and $N = 155$, and each $b_i$ and $A_{i, :}$ ($i = 1,\ldots, N$) are constructed as described in Figure~\ref{data_co}.
\begin{figure}[h!]
\centering
\includegraphics[width=1.0\linewidth]{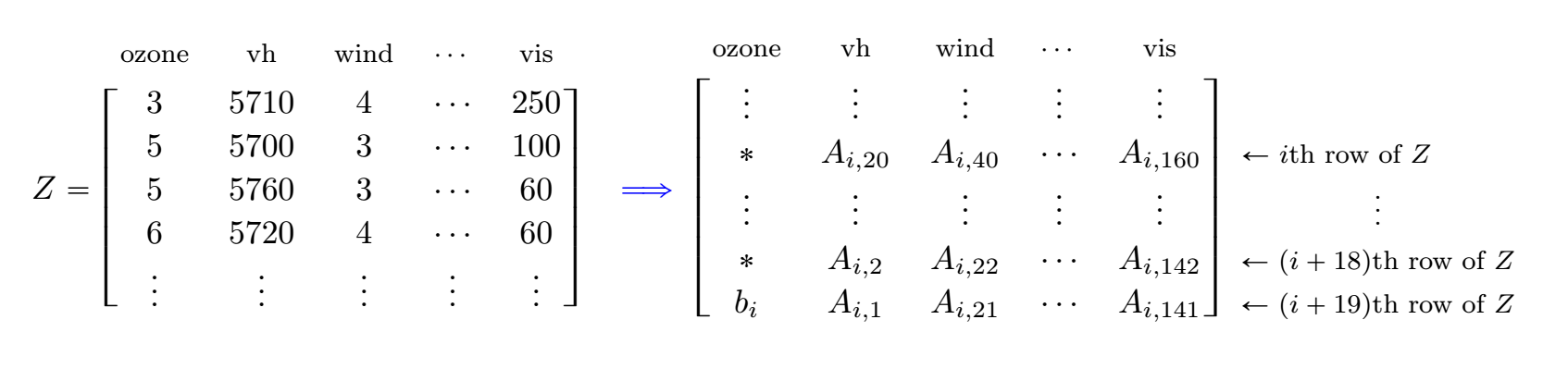}
\vspace{-1.1cm}
\caption{\small{Here, $Z$ represents the data matrix in the aforementioned link of real data, and the matrix on the right hand side presents the position (in $Z$) of $b_i$ and $A_{i, j}$ ($j = 1,\ldots, 160$), each of which takes the value of the element of $Z$ with the same position. For example, $b_2$ is set to be the ozone data in the 21st row of $Z$, and $A_{2, 1}, A_{2, 2},\ldots A_{2, 20}$ are set to be the vh data in the 21st, 20th, $\ldots$, and the 2nd row of $Z$, respectively.}
}
\label{data_co}
\end{figure}

As we can see from Figure~\ref{data_co}, the training data matrix $A\in\R^{155\times 160}$ corresponds to the data of 8 predictors (vh, $\ldots$ vis)
in $Z$  from row 1 to row 174, and $b\in\R^{155}$ correspond to the data of ozone in $Z$ from row 20 to row 174. We construct the validation data matrix $\widetilde{A}\in\R^{155\times 160}$ and $\widetilde{b}\in\R^{155}$ in a similar way as in Figure~\ref{data_co}, where on the right hand side of Figure~\ref{data_co}, the $A$ and $b$ are replaced by $\widetilde{A}$ and $\widetilde{b}$ respectively, and the row counter of $Z$ starts from $i+N$ instead of $i$. In essence, elements of $\widetilde{A}$ correspond to the data of 8 predictors in $Z$ from row 156 to row 329, and elements of $\widetilde{b}$ correspond to the data of ozone in $Z$ from row 175 to row 329.

Given that the data for the 8 predictors are measured on different scales, standardizations of each column of $A$ and $b$ are conducted before solving \eqref{order_blk}:
\begin{equation*}
{\bm a}_i \leftarrow \frac{{\bm a}_i - {\rm mean}({\bm a}_i)e}{{\rm std}({\bm a}_i)}, \ \ \ b\leftarrow \frac{b - {\rm mean}(b)e}{{\rm std}(b)},
\end{equation*}
where ${\bm a}_i$ is the $i$th column of $A$, and ${\rm mean}(\cdot)$ and ${\rm std}(\cdot)$ stand for the sample mean and the sample standard deviation, respectively. Once we solve \eqref{order_blk} with the standardized $A$ and $b$ as described above to obtain an approximate solution, say $x^*$, we will  predict $\widetilde{b}$ by
\begin{equation*}
\widetilde{b}_{\rm pred} = {\rm std}(b)\cdot(\widetilde{A}'x^*)  + {\rm mean}(b)e,
\end{equation*}
where $\widetilde A'$ is obtained from $\widetilde A$ by standardizing each column of $\widetilde A$.

Next, we will solve \eqref{order_blk} with {\color{blue}$q = 0.3$ and} $q = 0.5$ by Algorithm~\ref{Alg1} (\textbf{DMA}). In \cite{TibSuo16}, problem \eqref{order_blk} with $q = 1$ was approximated by a convex problem by replacing each block of constraints by the constraints as in \eqref{Prob11}.  As mentioned in the introduction, the solution obtained from this approximation model may lack proper interpretation.
Meanwhile, note that \eqref{order_blk} with $q = 1$ can be solved by \textbf{NPG} in view of Remark~\ref{remark:I}. In our experiments below, we will compare \textbf{DMA} with \textbf{NPG} (which solves \eqref{order_blk} with ${\color{blue}q} = 1$) in terms of validation error, which is defined by $\|\widetilde{b}_{\rm pred} - \widetilde{b}\|$.

\paragraph{Algorithm settings.}  For \textbf{DMA} and \textbf{NPG}, we generate the same random initial point $x^0\in\R^{pK}$ with each $K$-dimensional block having nonincreasing entries in the same way as described in Section~\ref{sec5-1}, and terminate both algorithms  whenever
\begin{equation*}
\frac{\|x^k - x^{k-1}\|}{\max\left\{1,\, \|x^k\|\right\}} < 10^{-6}.
\end{equation*}
The other parameters for  \textbf{DMA} and \textbf{NPG} are the same as in Section~\ref{sec5-1}. In Step 1b), the subproblems of these algorithms reduce to $p$ separate projection problems onto the set $\widehat{\Omega}:= \{y\in\R_+^K: y_1\ge\cdots\ge y_K \}$, which again will be solved by the solver developed from \cite{K64}.

In our test, for a sequence of $\lambda$ generated from the Matlab command
``logspace(-4, 1, 100)", we solve the corresponding \eqref{order_blk} by  \textbf{DMA} {\color{blue}(with $q = 0.3$ and $q = 0.5$)} and \textbf{NPG}, and  then compute their identification errors  (defined by $\|Ax^* - b\|$,  {\color{blue} denoted by DMA$_{\rm id}^{0.3}$, DMA$_{\rm id}^{0.5}$ and NPG$_{\rm id}$,  respectively}) and validation errors (defined by $\|\widetilde{b}_{\rm pred} - \widetilde{b}\|$, {\color{blue} denoted by DMA$_{\rm v}^{0.3}$, DMA$_{\rm v}^{0.5}$ and NPG$_{\rm v}$, respectively}). In Figure~\ref{real_err}, we first plot the  identification errors and validation errors with different $\lambda$ for \textbf{DMA} {\color{blue}(with $q = 0.3$ and $q = 0.5$)} and \textbf{NPG}. Next, for each algorithm, we select a proper $\lambda$ in the sense of simultaneously leading to small identification error and small validation error.
The one we select for \textbf{DMA} is {\color{blue}$\lambda = 3.68\times 10^{-3}$ when $q = 0.3$,} $\lambda = 4.13\times 10^{-3}$ {\color{blue}when $q = 0.5$,} and for \textbf{NPG}  is $\lambda = 1.67\times 10^{-2}$, which correspond to the $\lambda$ in Figure~\ref{real_err} {\color{blue}(the first three pictures)} that leads to the smallest validation error {\color{blue}DMA$_{\rm v}^{0.3}$ (55.55), DMA$_{\rm v}^{0.5}$ (56.17) and NPG$_{\rm v}$ (56.98), respectively. In view of this, \textbf{DMA} has a slightly better prediction that \textbf{NPG}.}
In the last picture of Figure~\ref{real_err}, we plot the predicted ozone concentration $\widetilde{b}_{\rm pred}$ for \textbf{DMA} {\color{blue}(with $q = 0.3$ and $q = 0.5$)} and \textbf{NPG} ({\color{blue} denoted by predicted$_{{\rm DMA}_{0.3}}$, predicted$_{{\rm DMA}_{0.5}}$ and predicted$_{\rm NPG}$ respectively}, each solves \eqref{order_blk} with the $\lambda$ selected above) and  true ozone concentration $\widetilde{b}$ {\color{blue} (denoted by true)}.
One can  see from the picture that the prediction from {\bf DMA} has fewer negative entries in the predicted ozone concentration: {\color{blue}4 negative entries from {\bf DMA} with $q = 0.5$ and 7 negative entries from {\bf NPG}.}

\begin{figure}[h!]
\begin{subfigure}{.5\textwidth}
\centering
\includegraphics[width=.95\linewidth]{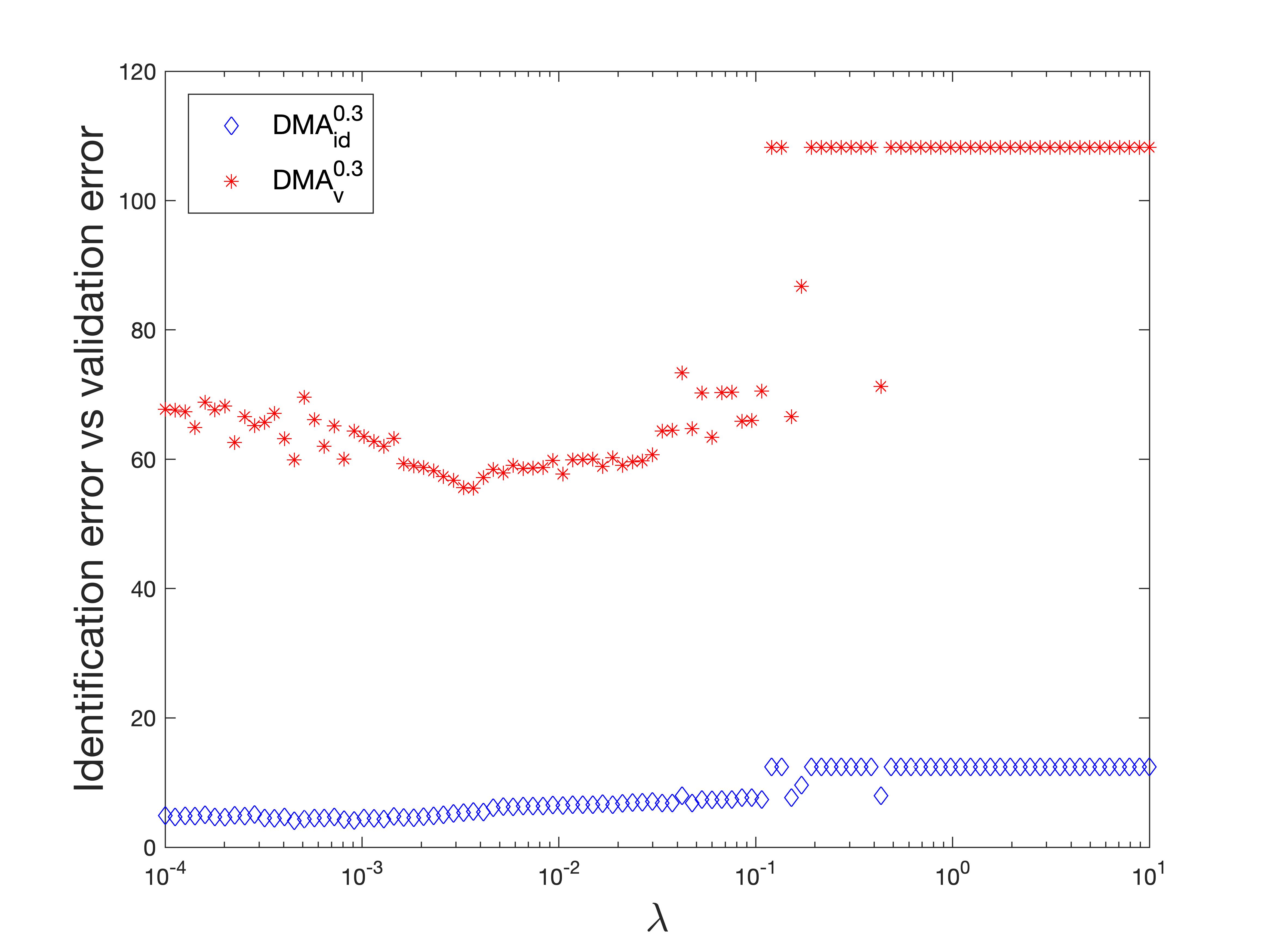}
\end{subfigure}
\begin{subfigure}{.5\textwidth}
\centering
\includegraphics[width=.95\linewidth]{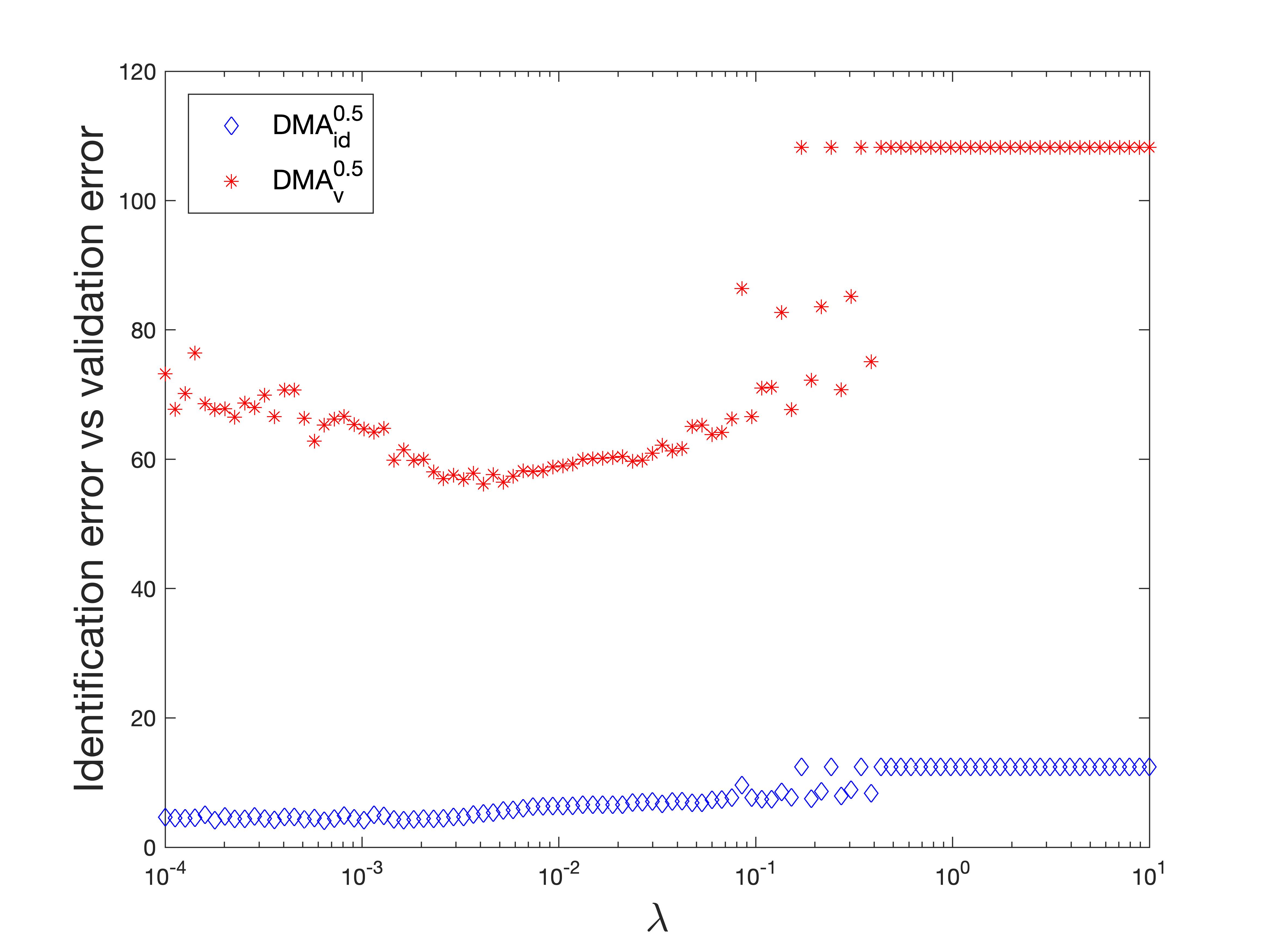}
\end{subfigure}
\begin{subfigure}{.5\textwidth}
\hspace{-9.9mm}
\includegraphics[width=.95\linewidth]{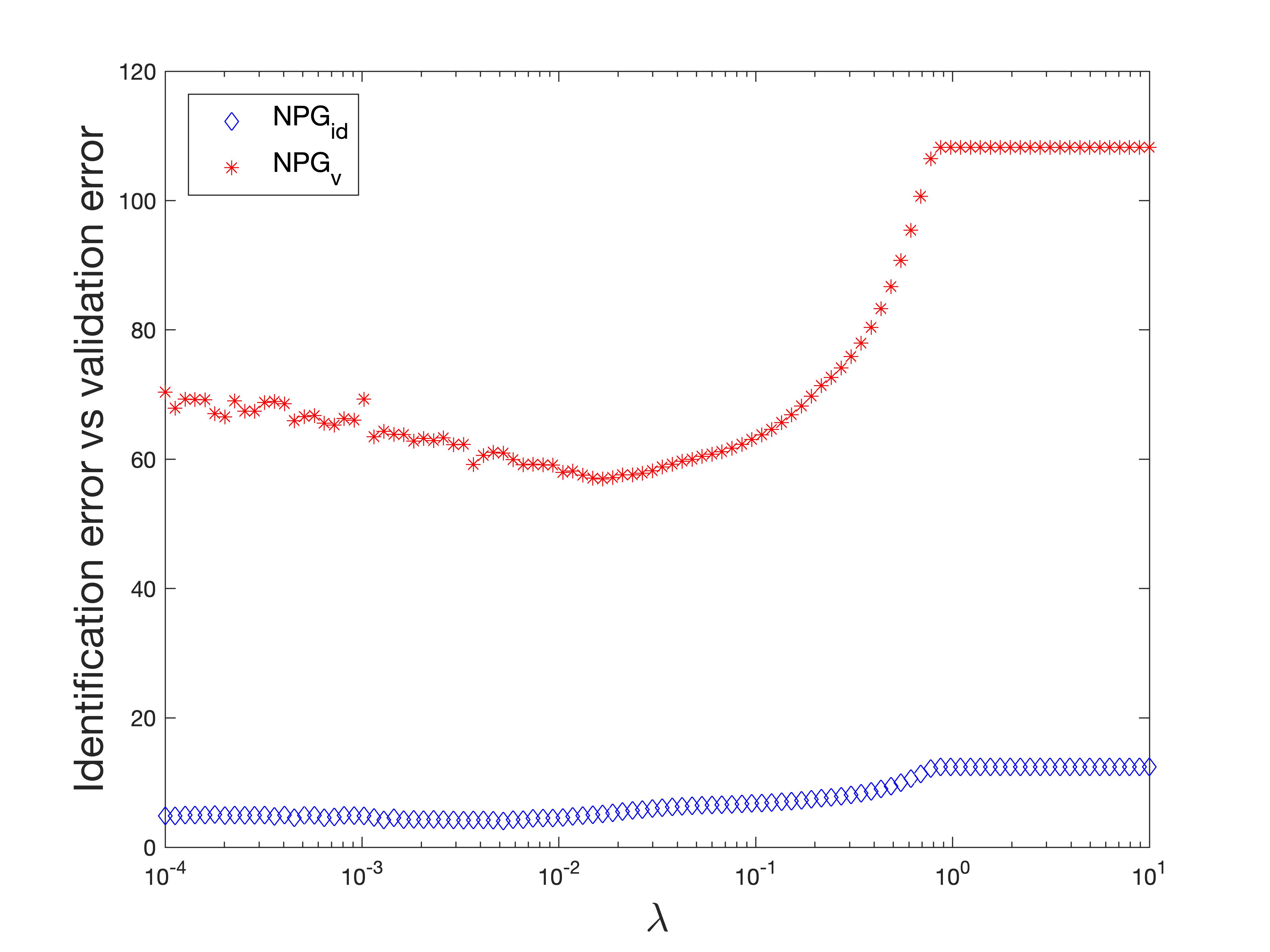}
\end{subfigure}
\hspace{-2cm}%
\begin{subfigure}{.5\textwidth}
\centering
\includegraphics[width=1.4\linewidth]{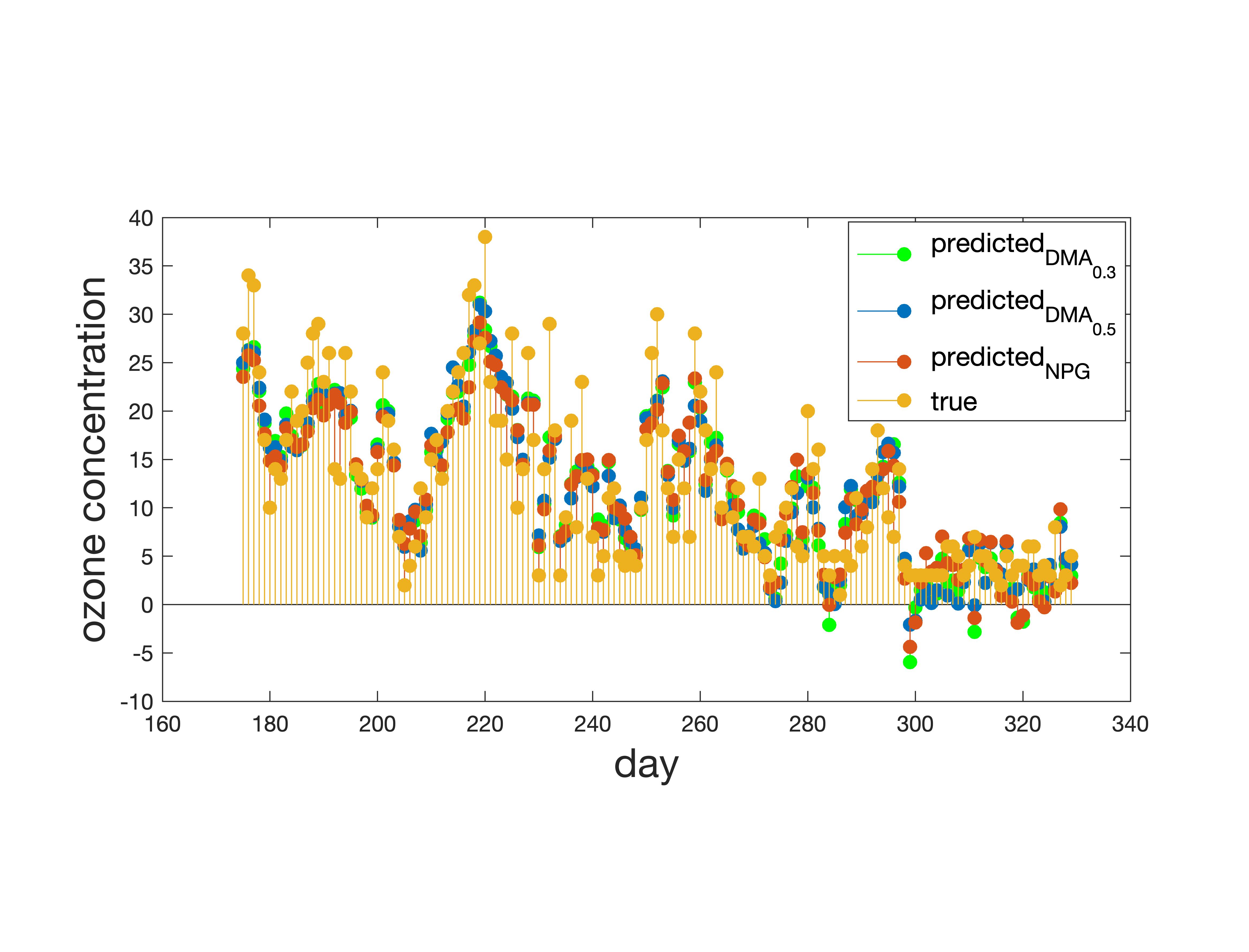}
\end{subfigure}
\vspace{-1.5cm}
\caption{{\color{blue}The first three pictures plot the identification error (blue point) and the validation error  (red point)  with different $\lambda$ for \textbf{DMA} (with $q = 0.3$ and $q = 0.5$) and \textbf{NPG}, respectively. The last picture presents the true ozone concentration and predicted  ozone concentration  for \textbf{DMA} with $\lambda = 3.68\times 10^{-3}$ when $q = 0.3$, $\lambda = 4.13\times 10^{-3}$ when $q = 0.5$ and \textbf{NPG} with $\lambda = 1.67\times 10^{-2}$.}}
\label{real_err}
\end{figure}

\end{document}